\documentclass[11pt]{article}
\usepackage{amsfonts}
\usepackage{amsmath}
\usepackage{amssymb}
\usepackage{graphicx}
\usepackage{algorithm}
\usepackage{algorithmic}
\usepackage{cases}
\usepackage{epstopdf}
\usepackage{geometry, tabularx}
\usepackage{multirow,makecell}
\usepackage{enumerate}
\usepackage{bm}
\usepackage{pifont}
\usepackage{footnote}
\usepackage{titlesec}
\usepackage[titletoc]{appendix}
\usepackage{setspace}

\usepackage{mathrsfs}

\providecommand{\U}[1]{\protect\rule{.1in}{.1in}}

\newtheorem{thm}{\bf Theorem}      
\newtheorem{cor}{\bf Corollary}[section]     

\newtheorem{prop}{\bf Proposition}
\newtheorem{defn}{\bf Definition}[section]

\newtheorem{rem}{\bf Remark}

\usepackage{hyperref}
\hypersetup{hypertex=true,
            colorlinks=true,
            linkcolor=blue,
            anchorcolor=green,
            citecolor=red}

\usepackage{graphicx,color}
\usepackage[bf,SL,BF]{subfigure}

\newenvironment{proof}[1][Proof]{\noindent\textbf{#1.} }{\ \rule{0.5em}{0.5em}}
\normalsize
\titlespacing*{\section} {0pt}{9pt}{0pt}
\numberwithin{equation}{section}
\allowdisplaybreaks

\usepackage{booktabs}

\begin{document}
\normalsize
\title{Multi-block linearized alternating direction method for sparse fused Lasso modeling problems}


\author{ 
Xiaofei Wu\thanks{College of Mathematics and Statistics, Chongqing University, Chongqing, 401331, P.R. China. Email: xfwu1016@163.com}, \ \
Rongmei Liang \thanks{Contributed equally as joint first author. College of Mathematical Medicine, Zhejiang Normal University, Jinhua, 321004, P.R. China. Email: rmliang@zjnu.edu.com.}, \ \ 
Zhimin Zhang\thanks{Corresponding author. College of Mathematics and Statistics, Chongqing University, Chongqing, 401331, P.R. China. Email: zmzhang@cqu.edu.cn} \ \ and \ \ Zhenyu Cui \thanks{School of Business, Stevens Institute of Technology, Hoboken, United States, NJ 07030. Email: zcui6@stevens.edu}}
\date{}
\maketitle
\vspace{-.50in}

\begin{abstract}
In many statistical modeling problems, such as classification and regression, it is common to encounter sparse and blocky coefficients. Sparse fused Lasso is specifically designed to recover these sparse and blocky structured features, especially in cases where the design matrix has ultrahigh dimensions, meaning that the number of features significantly surpasses the number of samples. Quantile loss is a well-known robust loss function that is widely used in  statistical modeling. In this paper, we propose a new sparse  fused lasso classification model, and develop a unified multi-block linearized alternating direction method of multipliers  algorithm that effectively selects sparse and blocky features for regression and classification. Our algorithm has been proven to converge with a derived linear convergence rate. Additionally, our algorithm has a significant advantage over existing methods for solving ultrahigh dimensional sparse fused Lasso regression and classification models due to its lower time complexity. Note that the algorithm can be easily extended to solve various existing fused Lasso models. Finally, we present numerical results for several synthetic and real-world examples, which demonstrate the robustness, scalability, and accuracy of the proposed classification model and algorithm.
\end{abstract}

\section{Introduction}
\textcolor{red}{Sparse and blocky (stepwise) features have been observed across various application fields. For instance, they have been extensively studied in the context of prostate cancer analysis by \cite{JFY2015,LLYPE2021}, where it is crucial to consider groups of interacting genes within each pathway rather than individual genes.
In image denoising, as studied by \cite{WYN,TH}, it is common practice to penalize differences between neighboring pixels to achieve smoothness and reduce noise.
In the study of networks by \cite{MP2023,ZZSW2023}, gene regulatory networks from developmentally closer (less distant) cell types exhibit greater similarities than those from more distant cell types.}
Similar situations also occur in comparative genomic hybridization (CGH) by \cite{MAE2018,WLY}, computer vision by \cite{MCH2021,GLY2024}, signal processing by \cite{M2019,LSFX2021}, community multiscale air quality by \cite{YHMF,D2021},  portfolio selection by \cite{CSM,MT2023}, among others.

To effectively extract the sparse and blocky features, Tibshirani et al. \cite{TSRZ} combined Lasso and total variation, proposing the sparse fused Lasso regularization term (SFL).  In addition, they first proposed SFL least squares (SFL-LS) regression and SFL support vector machine (SFL-SVM) classification, and introduced quadratic programming and linear programming techniques to solve SFL-LS and SFL-SVM, respectively. However, as mentioned by \cite{TSRZ},  ``one difficulty in using the sparse fused Lasso is the computational speed, and when $p > 2000$ and $n > 200$, speed could become a practical limitation." \textcolor{red}{This sentence illustrates that solving sparse fused Lasso models becomes challenging when dealing with a large number of features. To be specific, in the case of ultra-high dimensional data, many traditional and efficient algorithms, such as gradient descent and coordinate descent, are no longer suitable due to the presence of feature matrices in classification and regression, as well as the unsmooth and inseparable properties of sparse fused Lasso.}

Over the past two decades, advancements in computer technology have led to the development of several effective algorithms for high-dimensional LS-SFL regression, \textcolor{red}{including the efficient fused Lasso algorithm \cite{LYJ2011}, alternating linearization \cite{XMA2011}, the smoothing proximal gradient method \cite{XQSJ2012}, the majorization-minimization algorithm \cite{YWL}, alternating direction method of multipliers (ADMM) \cite{YX}, and linearized ADMM (LADMM) \cite{LMY}.  However, the above algorithms, except for ADMM and LADMM, only apply differentiable least squares losses and are no longer applicable to other regression and classification losses.} Recently, several ADMM and LADMM algorithms have been proposed to solve SFL regularized regression with various other losses, such as  square root loss \cite{JLD},  least absolute deviation loss \cite{LTZ}, quantile loss \cite{W2023}. \textcolor{red}{For the classification problem of SFL, currently only SFL-SVM (hinge loss plus SFL) in  \cite{TSRZ}  is available. Due to the non differentiability of its loss function and  the existence of SFL regularization, there has been limited research on calculating its high-dimensional situation, with only the ADMM algorithm proposed by  \cite{YX}. However, this algorithm is time-consuming, especially when $p$ is relatively large. In addition, there is a requirement for the convergence of this ADMM algorithm that the solution of SFL-SVM is unique, as pointed out by \cite{W2023}, which is impossible in high-dimensional situations. They also revealed that the condition for this uniqueness is redundant.} 

In this paper, we introduce a novel SFL classification model utilizing the pinball loss from \cite{HSS}, which extends the hinge loss and offers enhanced robustness. We demonstrate that some SFL  regression and classification models share a uniform optimization structure. Leveraging this unified form, we advocate for the utilization of the linearized ADMM algorithm for consistent solution. And the new algorithm proposed in this paper can converge without the need for the uniqueness of \cite{YX}.
 The main contributions of this paper are as follows:
\vspace{-1em}
\begin{enumerate}
\item We propose a pinball SVM with SFL (SFL-PSVM) to achieve robust classification modeling. This robust classification model not only inherits the resampling robustness of pinball loss, but also has the ability to select  sparse and blocky structured features in classification.

\item Sparse fused Lasso  has been applied to many distance-based losses  and margin-based losses  to realize classification  \cite{TSRZ,YX}  and regression  \cite{JLD,W2023}  modeling tasks, respectively. However, algorithms for solving these two tasks usually cannot be shared. In this paper, we point out that both quantile regressions and pinball SVMs with SFL  regularizations can be solved by minimizing piecewise linear optimization  in (\ref{unified}).  One can refer to Section \ref{sec31} for more details. Although our paper only focuses on SFL  regularization, this idea can easily be expended to other common  regularizations because it is flexible for regularizations; see Proposition \ref{prop1}. 

\item The most significant contribution is that we develop a multi-block linearized ADMM (MLADMM) to solve the unified  optimization form for SFL classification and regression. Previous studies have utilized multi-block ADMMs  to address SFL classification \cite{YX} and regression \cite{W2023}.  However, our MLADMM offers two key advantages over existing algorithms.
Firstly, our algorithm requires optimization of fewer primal variables, facilitating faster convergence and improved computational accuracy. Secondly, the previously proposed multi-block ADMM approaches necessitate the computation of the inverse of a large scale ($p \times p$) matrix, which can be time-consuming in cases involving ultrahigh dimensions.
Our algorithm circumvents this computation by incorporating linearized quadratic terms. Furthermore, we employ the power method, as suggested by \cite{L2023}, to compute the maximum eigenvalue of the positive definite matrix required by the linearized ADMM algorithm. Consequently, our multi-block linearized ADMM algorithm exhibits remarkable computational speed advantages over existing multi-block ADMM algorithms, which is also reflected in its algorithmic complexity.

Following \cite{W2023}, the unified optimization objective function can be transformed into a standard two-block ADMM similar to the approach presented in \cite{CBYY}.  Then, we can follow the works in \cite{BY,BY2}, and prove convergence and derive convergence rates of MLADMM. Furthermore, our algorithm proves to be versatile, and it can be extended to solve other SFL models, as demonstrated in \cite{W2023}. These encompass variants such as least squares sparse fused Lasso (SFL-LS), sparse fused Lasso SVM (SFL-SVM), sparse fused Lasso square root (SFL-SR), and sparse fused Lasso SVR (SFL-SVR). For more comprehensive details, please refer to the supplementary material \ref{B}.
\end{enumerate}

\textbf{Organization and Notations}:  In Section \ref{sec2}, we introduce some background knowledge  about classification and regression modeling, and review some existing algorithms for solving SFL problems. We present a unified optimization form of classification and regression, and propose multi-block linearized ADMM algorithm to solve it in Section \ref{sec3}. Convergence and computational cost analysis can also be referred to Section \ref{sec3}.  In Section \ref{sec6}, numerical results will show that the proposed method is scalable, efficient, and accurate even when existing methods are unreliable. Section \ref{sec7} summarizes the findings and concludes the paper with a discussion of future research directions. R codes for implementing MLADMM are available at \url{https://github.com/xfwu1016/LADMM-for-qfLasso}.

For $u \in \mathbb{R}$, let $u_{+}=u \textbf {I}(u>0) $ and  $u_{-}=-u \textbf {I}(u<0)$ be the positive and negative parts of $u$ respectively, where $\textbf {I}(  \cdot  )$ is the indicator function. Then, $\text{sign}(u) = {\textbf{I}}(\boldsymbol{u} > 0) - {\textbf{ I}}(\boldsymbol{u} < 0)$ is the sign function. Moreover, $\bf 0_n$ and $\bf 1_n$ respectively represent $n$-dimensional vector whose elements are all 0 and 1, $\bm I_n$ stands for $n$ identity matrix, and $\bm F$ is a $(p-1) \times p$ matrix with all elements being 0, except for 1 on the diagonal and -1 on the superdiagonal.

\section{Preliminaries and literature review}\label{sec2}
 Consider a prediction modeling problem with $n$ observations of the form 
\begin{equation}\label{data}
\{ {y_i},{\boldsymbol{x}_i}\} _i^n = \{ {y_i},{x_{i,1}},{x_{i,2}}, \cdots ,{x_{ip}}\} _i^n =\{\boldsymbol y, \boldsymbol{X} \} =\boldsymbol D
\end{equation}
which is supposed to be a random sample from an unknown joint (population) with probability density $f(\boldsymbol{x},{y})$.
The random variable ${y}$ is the ``response" or ``outcome", and $\boldsymbol{x}=\{{x}_1, {x}_2, \dots, {x}_p  \}$ are the predictor variables (features). These features may be the original observations and/or given functions constructed from them.  For convenience, we write $n$ observations of $\bm x$ into an $n \times p$ numerical matrix $\boldsymbol{X}$, which is often referred to as the design matrix in statistical modeling. Similarly, $n$  observations on variable $y$ are written into $n$-dimensional numerical vector $\boldsymbol y$. Without loss of generality, ${y}_i$ is quantitative for regression models, but equals to -1 or 1 for classification models. 

The purpose of classification or  regression  is to estimate the joint values for the parameters (coefficients) $\beta_0$ and $\boldsymbol{\beta}=(\beta_1, \beta_2, \dots, \beta_p)^\top$ of the linear model
\begin{equation}\label{linear model}
F(\boldsymbol{x}; \beta_0, \boldsymbol{\beta}) : = \sum\limits_{j = 1}^p \beta_j x_j+\beta_0
\end{equation}
for predicting $ y$ given $\boldsymbol x$, which minimize the expected loss (``risk")
\begin{equation}\label{risk}
\mathcal{R}(\beta_0,\boldsymbol{\beta}) = E_{\boldsymbol{x},{y}}\mathcal{L}({y},F(\boldsymbol{x};\beta_0,\boldsymbol{\beta}))
\end{equation}
over joint density function $f(\boldsymbol{x},{y})$. Here, $\mathcal{L}({y},F(\boldsymbol{x};\beta_0,\boldsymbol{\beta}))$ measures the cost of the discrepancy between the prediction $F(\boldsymbol{x};\beta_0,\boldsymbol{\beta})$ and the observation ${y}$ at the point $\boldsymbol{x}$.

For a specific loss $\mathcal{L}$, the optimal coefficients can be obtained from (\ref{risk}),
\begin{equation}\label{I1}
(\beta_0^*,\boldsymbol{\beta}^*) \in \mathop {\arg \min }\limits_{\beta_0,\boldsymbol \beta} \mathcal{R}(\beta_0,\boldsymbol{\beta}).
\end{equation}
Since the true joint density function $f(\boldsymbol{x},{y})$ is unknown, a popular practice is to replace an empirical estimate of the expected loss in (\ref{risk}) based on the collected data $\boldsymbol D$, yielding
\vspace{-0.5em}
\begin{equation}\label{I2}
(\hat \beta_0,\hat{\boldsymbol{\beta}}) \in \mathop {\arg \min }\limits_{\beta_0,\boldsymbol \beta} {\hat{\mathcal{R}}}(\beta_0,\boldsymbol{\beta}) = \mathop {\arg \min }\limits_{\beta_0,\boldsymbol \beta}\frac{1}{n}\sum\limits_{i = 1}^n \mathcal{L}(y_i, \sum\limits_{j = 1}^p x_{ij}\beta_j+ \beta_0)
\vspace{-0.5em}
\end{equation}
as an estimator for $\beta_0^*$ and $\boldsymbol{\beta}^ {*}$ in (\ref{I1}).
In this paper, the variables are assumed to be sparse and blocky in the sense that the coefficients come in blocks and have only a few change points. To be precise, we can partition $\{1,2,...,p\}$ into $J$ groups such that $\{L_1,...,U_J\}$ form a partition and $L_1=1$, $U_J=p$, $U_j\geq L_j$, $L_{j+1}=U_j+1$, and define the stepwise function as
$\boldsymbol{\beta}^*=\sum_{j=1}^J\beta_{L_j}^*\boldsymbol{1}_{\{L_j,...,U_j\}}$,
where $\boldsymbol{1}_{\{L_j,...,U_j\}}$ is a $p$-dimensional vector with ones only at positions corresponding to $\{L_j,...,U_j\}$ and zeros elsewhere. Besides, we assume that the vector $(\beta_{L_1}^*,...,\beta_{L_J}^*)^\top$ is \textit{sparse}, meaning that only a few elements of $\boldsymbol{\beta^*}$ are nonzero.

In order to accurately estimate these sparse and blocky coefficients, empirical risk with  (abbreviated as SFL) is a reliable model, that is, 
\vspace{-0.5em}
\begin{equation}\label{RFL}
 \mathop {\arg \min }\limits_{\beta_0,\boldsymbol \beta} \frac{1}{n}\sum\limits_{i = 1}^n \mathcal{L}(y_i, \sum\limits_{j = 1}^p x_{ij}\beta_j+ \beta_0) + {\lambda _1}\sum\limits_{j = 1}^p {\left| {{{\beta} _j}} \right|}  + {\lambda _2}\sum\limits_{j = 2}^p {\left| {{{\beta} _j} - {{\beta} _{j - 1}}} \right|}.
\vspace{-0.5em}
\end{equation}
The first regularization term (Lasso \cite{T}) with parameter $\lambda_1 \ge 0$ encourages the sparsity of the coefficients, while the second regularization term (total variation \cite{RO}) with parameter $\lambda_2 \ge 0$ shrinks the differences between neighboring coefficients towards zero. This means that it can achieve simultaneous sparseness and blockiness of coefficients. In this paper, we are all concerned with regression and classification problems. Let us begin with the following basic definition in \cite{IA2008}, which introduces the forms of loss functions used in above two problems.
\vspace{-1em}
\begin{defn}\label{def1}
We say that a supervised loss is 
\vspace{-1em}
\item 1. \textit{distance-based} if there exists a \textit{representing function} $\psi : \mathbb{R} \rightarrow [0, \infty )$ satisfying $\psi(0)=0$ and  
\vspace{-0.5em}
\begin{equation}\label{ls}
\mathcal{L}(y,F)= \psi(y-F) \ \  y \in \mathbb{R}; 
\end{equation}
\vspace{-3em}
\item 2. \textit{margin-based} if there exists a \textit{representing function}  $\psi : \mathbb{R} \rightarrow [0, \infty )$ such that
\vspace{-0.5em}
\begin{equation}\label{ls}
\mathcal{L}(y,F)= \psi(yF), \ y \in \{-1,1\};
\end{equation}
\vspace{-3em}
\item 3.  symmetric if $\mathcal{L}$ is distance-based and its representing function $\psi$ satisfies 
\vspace{-0.5em}
\begin{equation}\label{ls}
\psi(r)= \psi(-r), \ r \in \mathbb{R}.
\end{equation}
\end{defn}
\vspace{-1em}
Both distance-based and margin-based losses have good properties, one can refer to Lemma 2.25 and 2.33 in \cite{IA2008}. From the above definition, we can easily see that distance-based losses are used for regression and margin-based losses for classification. Throughout the paper, $F$ represents $F(\boldsymbol x; \boldsymbol \beta)$ in (\ref{linear model}).
\vspace{-1em}
\subsection{Distance-based losses for regression}\label{sec21}
\vspace{-0.5em}
In linear regression, the problem is to predict a real-valued output ${{y}}$ given an input $\boldsymbol{x}$. The discrepancy between the prediction $F$ and the observation ${{y}}$ is always measured by the distance-based losses. Many distance-based losses have been used in (\ref{I1}) to implement regression. For historical and computational reasons, the most popular of which is least squares (LS) loss 
\begin{equation}\label{ls}
\mathcal{L}_{LS}(y,F)= \psi(y-F)= ({y}-F)^2, \ \  y \in \mathbb{R}. 
\end{equation}
It is well-known that the corresponding $F$ is the conditional mean of $\boldsymbol y$ given $\boldsymbol x$.

However, from a theoretical and practical point of view, there are many situations in which a different loss criterion is more appropriate. For example, when heavy-tailed errors or outliers exist in response $ y$, least absolute deviation (LAD) loss is more applicable than LS loss. LAD loss defined by
\vspace{-0.5em}
\begin{equation}\label{lad}
\mathcal{L}_{LAD}(y,F)=|{y}-F|, \ \ y \in \mathbb{R},
\end{equation}
is the first reasonable surrogate for LS.
Here, $F$ is actually not interested in conditional mean but in the conditional median instead. Koencker \cite{K2005} first expanded LAD loss and proposed quantile loss,
\begin{equation}\label{quantile}
{\rho _\tau }(y-F) = (y-F)\{ \tau  - \textbf{I}(y-F < 0)\},\ \ \tau \in [0,1] \ \text{and} \ y \in \mathbb{R}.
\end{equation}
It is obvious that here $F$ focuses on conditional quantile. Quantile regression has been widely used in many fields, such as econometrics \cite{K2005}, machine learning and statistics \cite{IA2008}. Besides, there are many other distance-based losses (see Section 2.4 in \cite{IA2008}) that can be used for regression. The LS, LAD, and square root (SR) losses, plus sparse fused Lasso in (\ref{RFL}), were proposed by \cite{TSRZ}, \cite{LTZ}, and \cite{JLD}, respectively, referred to as SFL-LS, SFL-LAD, and SFL-SR. Recently, Wu et al. \cite{W2023} proposed sparse fused Lasso quantile regression (SFL-QRE) by combining quantile loss and sparse fused Lasso in (\ref{RFL}).
\vspace{-1em}
\subsection{Margin-based losses for classification}\label{sec22}
\vspace{-0.5em}
In linear classification, the problem is to predict label output ${{y}}$ given an input $\boldsymbol{x}$. The discrepancy of signs between the prediction $\text{sign}(F)$ and the label observation ${{y}}$ is often measured by classification loss, which is defined as follows
\begin{equation}\label{class}
\mathcal{L}_{class}(y,F)= \textbf{I}_{ (-\infty,0] }(y  \text{sign} \ F), \ \ y \in \{-1,1\}.
\end{equation} 
Note that classification loss only penalizes predictions $F$ whose signs disagree with that of $y$, so it indeed reflects classification goal. However, classification loss is non-convex, and as a consequence solving its empirical risk is often NP-hard.

To resolve the issue, many alternative margin-based convex losses have been proposed to achieve classification task, such as hinge, squared hinge, LS, $q$-norm loss ($q > 2$). The most famous one among these surrogate losses is hinge loss
\begin{equation}\label{hinge}
\mathcal{L}_{hinge}(y,F)=\max(1-yF,0), \ \ y \in \{-1, 1\}.
\end{equation}
As showed by Theorem 2.31 (\textit{Zhang's inequality}) of \cite{IA2008}, hinge loss used in (\ref{I1}) for classification is a good surrogate because the excess risk of hinge loss is the upper bound of that of classification loss.
In other words, we can achieve accurate classification whenever the excess risk of hinge loss is small. The hinge loss with sparse fused Lasso in (\ref{RFL}), named SFL-SVM, was first proposed in \cite{TSRZ}.

However, hinge loss is related to the shortest distance between dataset $\boldsymbol X$ and the corresponding \textit{separating hyperplane} $F$, and thereby is sensitive to noise and unstable for resampling. For a specific explanation of this assertion, the reader can refer to \cite{HSS}. To solve this issue, the authors in \cite{HSS} suggested using the pinball loss instead of hinge loss to achieve classification. Margin-based pinball loss is defined as follows
\begin{align}\label{pinball}
{\mathcal{L}_{\tau}}(1 - yF) = \left\{ \begin{array}{l}
1 - yF, \ \ \ \ \ 1 - yF  \ge 0, \\
\tau (yF-1), \ 1 - yF  < 0.
\end{array} \right.
\end{align}
Note that the pinball loss reduces to hinge loss provided that $\tau=0$. The above pinball loss is related to quantile distance which is insensitive to the noise around the separating hyperplane, and hence robust to resampling. 
Furthermore, all margin-based losses mentioned in this section for classification have been proved to have an upper bound that can control the classification risk similar to \textit{Zhang's inequality}; see \cite{HSS} and Section 3.4 in \cite{IA2008}. In this paper, we will combine pinball loss and sparse fused Lasso together in (\ref{RFL}) to propose a new SVM called SFL-PSVM.
\vspace{-1em}
\subsection{Existing algorithms for ultrahigh dimensional  fused Lasso problems}\label{sec22}
\vspace{-0.5em}
The sparse fused lasso regularization (SFL)
\begin{align}\label{SFL}
{\lambda _1}\sum\limits_{j = 1}^p {\left| {{{\beta} _j}} \right|}  + {\lambda _2}\sum\limits_{j = 2}^p {\left| {{{\beta} _j} - {{\beta} _{j - 1}}} \right|} = {\lambda _1}\|\bm \beta \|_1 +  \lambda _2\|\bm F \bm \beta \|_1,
\end{align}
poses significant challenges in algorithm design due to its nonsmooth, nondifferentiable, and nonseparable properties, particularly when dealing with ultrahigh dimensional design matrices.
In this paper, we focus on the piecewise linear loss in (\ref{quantile}) and (\ref{pinball}) with SFL to complete the modeling task. Thus, we review two related  ADMMs for solving the piecewise linear loss with SFL, one is the multi-block ADMM algorithm proposed by Ye and Xie  \cite{YX} to solve SFL-SVM, and the other is the multi-block ADMM algorithm proposed by Wu et al. \cite{W2023} to solve SFL-QRE.  These two-multi block ADMMs both introduced the two slack variables $\bm a= \bm \beta$ and $\bm b=\bm{F\beta}$ to ensure that each subproblem of the multi-block ADMM algorithm has a closed-form solution.
To be more specific, these two slack variables can transform the SFL in  (\ref{SFL}) into $\lambda_1 \|\bm a\|_1+ \lambda_2 \|\bm b\|_1$, then terms $\lambda_1 \|\bm a\|_1$ and $\lambda_2 \|\bm b\|_1$ are completely separable. 
This allows for the implementation of the ADMM algorithm by alternately minimizing $\bm a$ and $\bm b$.  \textcolor{red}{However, the use of slack variables poses two issues: the need to calculate the inverse of a  ultrahigh dimensional  matrix ($p \times p$)  and the introduction of excessive slack variables ($3p+n$) in the ADMM algorithms. To address these two issues,  this paper proposes the multi-block linearized ADMM algorithm to uniformly solve  SFL classification and regression problems.
In our algorithm, first, we replace the inverse of an ultrahigh-dimensional matrix by calculating its maximum eigenvalue. Second, we eliminate the need to introduce $\bm a = \bm \beta$, reducing the number of subproblems to solve to $2p+n$ in each iteration. These two points not only reduce the computational burden, but also make the algorithm converge faster.}

On the other hand, for the study of SFL, there are many models and algorithms that consider when $\bm X$ is the identity matrix. Interested readers can refer to \cite{CDS,SSP}. Because the primary focus of this paper is on modeling high-dimensional data, we will briefly skip over this aspect.
\section{Multi-block Linearized ADMM Algorithm}\label{sec3}
In this section, we introduce a unified form of SFL regression and classification, and provide a multi-block linearized ADMM algorithm to solve it.
\vspace{-0.75em}
\subsection{Unified Form of Regression and Classification}\label{sec31}
\vspace{-0.75em}
For the collected data $\boldsymbol D=\{ {y_i},{x_{i,1}},{x_{i,2}}, \cdots ,{x_{ip}}\} _i^n $, sparse fused Lasso quantile regression (SFL-QRE) in \cite{W2023} is defined as follows
\vspace{-0.5em}
\begin{align}\label{qflasso}
(\hat \beta_{0r},\hat{\boldsymbol{\beta}}_r) \in \mathop {\arg \min }\limits_{(\beta_0,\boldsymbol{\beta}) \in \mathbb{R}^{p+1}} \left\{\hat{\mathcal{R}}_r(\beta_0,\boldsymbol{\beta} )   +  {\lambda _1}\sum\limits_{j = 1}^p {\left| {{{\beta} _j}} \right|}  + {\lambda _2}\sum\limits_{j = 2}^p {\left| {{{\beta} _j} - {{\beta} _{j - 1}}} \right|}\right \},
\vspace{-0.5em}
\end{align}
where $\hat{\mathcal{R}}_r(\beta_0,\boldsymbol{\beta} )=\frac{1}{n}\sum_{i = 1}^n {{\rho _\tau }} ({{y}_i} - \sum_{j = 1}^p  x_{i,{j}}\beta_j-\beta_0)$. Formula (\ref{qflasso}) is composed of distance-based quantile empirical risk and SFL penalty term. 

Similarly, classification model considered in this paper is defined as 
\vspace{-0.3em}
\begin{align}\label{qfsvm}
(\hat \beta_{0c},\hat{\boldsymbol{\beta}}_c) \in \mathop {\arg \min }\limits_{(\beta_0,\boldsymbol{\beta}) \in \mathbb{R}^{p+1}} \left\{\hat{\mathcal{R}}_c(\beta_0,\boldsymbol{\beta} )   +  {\lambda _1}\sum\limits_{j = 1}^p {\left| {{{\beta} _j}} \right|}  + {\lambda _2}\sum\limits_{j = 2}^p {\left| {{{\beta} _j} - {{\beta} _{j - 1}}} \right|}\right \},
\vspace{-0.75em}
\end{align}
where $\hat{\mathcal{R}}_c(\beta_0,\boldsymbol{\beta} )=\frac{1}{n}\sum_{i = 1}^n {{\mathcal{L}_\tau }} (1 - \sum_{j = 1}^p y_i x_{ij}\beta_j-y_i\beta_0)$.
Let us now relate the optimization problem (\ref{qfsvm}) to the SFL-SVM formulation in  \cite{TSRZ}. To this end, we need to introduce slack variables $\{\xi_i\}_{i=1}^n \ge 0$ which satisfies
\vspace{-0.75em}
\begin{equation}
\xi_i \ge 1- \sum_{j=1}^p y_i x_{ij}\beta_j - y_i\beta_0 \ \text{and} \ \xi_i \ge \tau(\sum_{j=1}^p y_ix_{ij}\beta_j + y_i\beta_0-1).
\vspace{-0.5em}
\end{equation}
Then, we can see that the slack variables must satisfy
\vspace{-0.6em}
\begin{equation}\label{slacksvm}
\xi_i \ge \max \left\{  1- \sum_{j=1}^p y_i x_{ij}\beta_j - y_i\beta_0, \tau(\sum_{j=1}^p y_i x_{ij}\beta_j + y_i\beta_0-1)\right\}=\mathcal{L}_\tau(1 - \sum_{j = 1}^p y_ix_{ij}\beta_j - y_i\beta_0),
\vspace{-0.6em}
\end{equation}
where $\mathcal{L}_\tau$ is the pinball loss introduced earlier. Obviously, $\sum_{i=1}^{n}\xi_i$ becomes minimal if the inequalities in (\ref{slacksvm}) are actually equalities. Then, the optimization problem (\ref{qfsvm}) can be rewritten as
\begin{equation}\label{pinsvm2}
\begin{array}{l}
\mathop {\min }\limits_{(\beta_0,\boldsymbol{\beta})\in \mathbb{R}^{p+1},{\xi} \ge \boldsymbol 0}  \frac{1}{n}\sum\limits_{i = 1}^n {{{\xi} _i}} +  {\lambda _1}\sum\limits_{j = 1}^p {\left| {{{\beta} _j}} \right|}  + {\lambda _2}\sum\limits_{j = 2}^p {\left| {{{\beta} _j} - {{\beta} _{j - 1}}} \right|} \\
\text{s.t.}\;\; \sum\limits_{j = 1}^p y_ix_{ij}\beta_j + y_i\beta_0  \ge 1 - {{\xi} _i},\quad i = 1,2,\cdots n,\\
\ \ \ \ \ \  \sum\limits_{j = 1}^p y_ix_{ij}\beta_j + y_i\beta_0  \le 1 + {1 \over \tau}{{\xi} _i}, \ i = 1,2,\cdots n.
\end{array}
\end{equation}
Note that (\ref{pinsvm2}) reduces to SFL-SVM in  \cite{TSRZ} with $\tau=0$. Thus, we call it  sparse fused Lasso pinball SVM  (abbreviated as SFL-PSVM).

Before introducing the unified form, we need to introduce the equivalence relationship between $\hat{\mathcal{R}}_r(\beta_0,\boldsymbol{\beta} ) $ and $\hat{\mathcal{R}}_c(\beta_0,\boldsymbol{\beta} )$.
Let $\bar{\bm X}_r = \left[ \bm 1_n, \bm X \right]$, and $\bar {\bm{X}}_c = \bm Y \bar {\bm X}_r$ where $\bm Y$ is a diagonal matrix with its diagonal elements to be the vector $\bm y$.
Then,  $\hat{\mathcal{R}}_r(\beta_0,\boldsymbol{\beta} ) $ and $\hat{\mathcal{R}}_c(\beta_0,\boldsymbol{\beta} )$  can be rewritten in the following forms without intercept term
\vspace{-0.75em}
\begin{align}
\hat{\mathcal{R}}_r(\beta_0,\boldsymbol{\beta} )=\frac{1}{n}\sum_{i = 1}^n {{\rho _\tau }} ({{y}_i} - \sum_{j = 0}^p \bar x^{'}_{i,{j+1}}\beta_j) \quad \text{and} \quad
\hat{\mathcal{R}}_c(\beta_0,\boldsymbol{\beta} )=\frac{1}{n}\sum_{i = 1}^n {{\mathcal{L}_\tau }} (1 - \sum_{j = 0}^p \bar x^{''}_{i,{j+1}}\beta_j),
\vspace{-0.85em}
\end{align}
where $\bar x^{'}_{i,{j+1}}$ and  $\bar x^{''}_{i,{j+1}}$ correspond to the  entry $\left(i,j+1\right)$  of the matrices $\bar{\bm X}_r$ and $\bar{\bm X}_c$  respectively.

It is clear that quantile loss and pinball loss can be transformed into each other through simple transformation, that is 
\vspace{-0.2em}
\begin{equation}\label{22}
\mathcal{L}_{\tau}(u) =u\{1 - (\tau+1){\textbf{I}} (u<0) \}=(\tau+1)\rho_{\tilde \tau}(u),
\vspace{-0.2em}
\end{equation} 
where $\tilde{\tau}=1/(1+\tau)$. Then,  
\vspace{-0.5em}
\begin{equation}
\hat{\mathcal{R}}_c(\beta_0,\boldsymbol{\beta} ) = \frac{1}{n}\sum_{i = 1}^n {{\mathcal{L}_\tau }} (1 - \sum_{j = 0}^p \bar x^{''}_{i,{j+1}}\beta_j)=\frac{\tau+1}{n}\sum\limits_{i = 1}^n {{\rho_{\tilde \tau} }} (1 - \sum_{j = 0}^p \bar x^{''}_{i,{j+1}}\beta_j).
\vspace{-0.5em}
\end{equation}
Take $ \tilde{y}_i = 1$, $\tilde{{x}} _{i,{j+1}} = {\bar{x}^{''}_{i,{j+1}}}$, $\tilde{\lambda}_1= \lambda_1/(\tau+1)$ and $\tilde{\lambda}_2= \lambda_2/(\tau+1)$, then formula (\ref{qfsvm}) can be rewritten as
\begin{align}\label{qqff}
(\hat \beta_{0c},\hat{\boldsymbol{\beta}}_c) \in \mathop {\arg \min }\limits_{(\beta_0,\boldsymbol{\beta}) \in \mathbb{R}^{p+1}} \left\{\hat{\mathcal{R}}_r(\beta_0,\boldsymbol{\beta} )   +  {\tilde{\lambda} _1}\sum\limits_{j = 1}^p {\left| {{{\beta} _j}} \right|}  + {\tilde{\lambda} _2}\sum\limits_{j = 2}^p {\left| {{{\beta} _j} - {{\beta} _{j - 1}}} \right|}\right \},
\end{align}
where $\hat{\mathcal{R}}_r(\beta_0,\boldsymbol{\beta}) = \frac{1}{n}\sum_{i = 1}^n {\rho _{\tilde{\tau} }} ({\tilde{y}_i} - \sum_{j = 0}^p \tilde{x}_{i,{j+1}}\beta_j  )$. Before solving $\hat \beta_{0c}$ and $\hat{\boldsymbol{\beta}}_c$, the values of ${\tilde{y}_i}$ and $\tilde{{x}} _{i,{j+1}}$ have been determined as they are input values.
Therefore, the  regression in (\ref{qflasso}) and the classification in (\ref{qfsvm}) have a unified optimization form in (\ref{qqff}). 
\textcolor{red}{Note that both quantile loss and pinball loss are convex, and SFL is also convex. Therefore, this unified optimization form is convex.}
Note also that if SFL regularization term in (\ref{qfsvm}) is replaced by other regularization terms, the above conclusion still holds. To sum up, we get the following proposition.
\begin{prop}\label{prop1}
For all $\hat \beta_{0c}$ and $\hat{\boldsymbol{\beta}}_c \in \mathop {\arg \min }\limits_{(\beta_0,\boldsymbol{\beta}) \in  \mathbb{R}^{p+1}} \left\{ \hat{\mathcal{R}}_c(\beta_0,\boldsymbol{\beta} )+ \text{regularization} \right \}$ with  $\hat{\mathcal{R}}_c(\beta_0,\boldsymbol{\beta} ) =\frac{1}{n}\sum\limits_{i = 1}^n {{\mathcal{L}_\tau }} (1 - \sum\limits_{j = 0}^p \bar x^{''}_{i,{j+1}}\beta_j)$, one can get $\hat \beta_{0c}$ and $\hat{\boldsymbol{\beta}}_c $  by solving \textcolor{red}{the  following convex optimization form}, $$\mathop {\min }\limits_{(\beta_0,\boldsymbol{\beta}) \in  \mathbb{R}^{p+1}} \left\{ \hat{\mathcal{R}}_r(\beta_0,\boldsymbol{\beta} )+ \tilde{\tau} \times \text{regularization} \right \}$$ with $\hat{\mathcal{R}}_r(\beta_0,\boldsymbol{\beta} ) =  \frac{1}{n}\sum\limits_{i = 1}^n {\rho _{\tilde{\tau} }} ({\tilde{y}_i} - \sum\limits_{j = 0}^p \tilde{x}_{i,{j+1}}\beta_j )$, where  $\tilde{\tau}=1/(1+\tau)$, $ \tilde{y}_i = 1$, and $\tilde{{x}} _{i,{j+1}} = {\bar{x}^{''}_{i,{j+1}}}$.
\end{prop}

\textcolor{red}{Proposition \ref{prop1} theoretically indicates that the algorithms of various pinball-SVM classification models and quantile regression models can be shared. Let us close this subsection with one \textcolor{red}{practical}  application of Proposition \ref{prop1},  where certain pinball-SVM classification models can be solved using existing quantile regression algorithms.}
Recently, some machine learning papers have focused on $\hat{\mathcal{R}}_c(\beta_0,\boldsymbol{\beta} )$ with regularization terms,
such as $\ell_2$ in \cite{HSS}, $\ell_1$ in \cite{TSR} and elastic net in \cite{WY2021}. Some algorithms have been proposed to solve the above three pinball-SVMs. For example, the authors in \cite{HSS,HSS2} used quadratic programming and sequential minimal optimization (SMO) algorithms to solve $\ell_2$ pinball-SVM. In \cite{TSR}, the authors utilized linear programming to solve $\ell_1$ pin-SVM. The authors in \cite{WY2021} proposed feature and label-pair elimination rule (FLER) to solve elastic net pin-SVM. However, these algorithms are not suitable for the case where $p$ is particularly large. Fortunately, statistics has done a lot of researches on  coordinate descent algorithms \cite{PW,YH} and ADMM  algorithms \cite{WLY,GFK}  for solving high-dimensional quantile regression. Among them, the most famous one is ADMM algorithms for solving $\ell_1$ and elastic net quantile regression in \cite{GFK}.
Proposition \ref{prop1} points out that the above three pinball SVMs can be effectively solved by the elastic net quantile regression algorithms because both $\ell_1$ and $\ell_2$ are special cases of elastic net. Besides, the authors in \cite{GFK} provided  \textbf{R} package called \textbf{FHDQR} to implement their algorithms.

\subsection{Linearized ADMM Algorithm for Unified Form}\label{sec32}
Proposition \ref{prop1} has shown that SFL-QRE and SFL-PSVM have a unified optimization objective function,
\vspace{-0.5em}
\begin{equation}\label{unified}
\mathscr{L}(\boldsymbol\beta ) = \frac{1}{n}\sum\limits_{i = 1}^n {{\rho _\tau }} ({\bar{y}_i} - \sum\limits_{j = 0}^p  \bar x_{i,{j+1}}\beta_j) + {\lambda _1}\sum\limits_{j = 1}^p {\left| {{{\beta} _j}} \right|} + {\lambda _2}\sum\limits_{j = 2}^p {\left| {{{\beta} _j} - {{\beta} _{j - 1}}} \right|}.
\vspace{-0.5em}
\end{equation}
The observation values of $\bar y_i$ and $\bar x_{i,{j+1}}$ can be determined by some simple transformations of  original observation dataset $\bm D$, and let ${\left\{\bar y_i,\bar x_{i,{j}} \right\}_{i=1}^{n}}_{j=0}^{p}=\left\{\bm {\bar y}, \bm{\bar X} \right\}$. For regression, $\bm {\bar y}= \bm y$ and $\bm {\bar X}= \bm {\bar X}_r$;  but for classification, $\bm {\bar y}= \bm 1_n$ and $\bm {\bar X}= \bm {\bar X}_c$.
When $p$ is small and medium scale, linear programming (LP) has been proved to be an efficient method to solve SFL classification and regression by many literatures; see \cite{WLY,TSRZ}. Ye and Xie \cite{YX} proposed a multi-block ADMM to solve ultrahigh dimensional SFL-SVM.  More recently, Wu et al. \cite{W2023} proposed using the multi-block ADMM to solve ultrahigh dimensional SFL-QRE. However, their ADMM algorithms  needs to update $n+3p-1$ primal variables require computing the inverse of a $p \times p$ matrix in the $\bm \beta$ iteration, which is a significant computational burden when $p$ is ultrahigh  dimensional. In this section, we are dedicated to developing a MLADMM algorithm to solve unified optimization form when the dataset is ultrahigh dimensional. The MLADMM algorithm not only requires updating fewer primal variables, but also avoids calculating the inverse of a large-scale matrix by adding linearization terms. This linearization significantly reduced the complexity of the ADMM algorithm for solving SFL problems  and improves computational speed in \cite{LMY}.

Next, we  develop the MLADMM algorithm for solving the unified form.  Precisely speaking, $\bar{\bm y}$ is $\bm y$ and $\bar {\bm X}$ is $[\bm 1_n,\bm X]$ for regression; while $\bar{\bm y}$ is $\bm 1_n$ and $\bar{\bm X}$ is $[\bm y, \bm {Y X}]$ for classification. For convenience,  we just take $\bar {\bm X}= [\bm \gamma, \bm X]$. 
Then, two  additional auxiliary variables are defined as $ {\bm X}\bm \beta + \bm \gamma \beta_0 + \boldsymbol{r} =\bar{\bm y}$ and $\boldsymbol{F\beta=b}$, where $\boldsymbol F$ is a $(p-1) \times p$ matrix with $0$ everywhere except $1$ in the diagonal and $-1$ in the superdiagonal. Adding these additional auxiliary variables, the optimal solution of (\ref{unified}) can be reformulated into
\begin{equation}\label{qflasso_con1}
\begin{array}{l}
\mathop {\min }\limits_{\beta_0, \boldsymbol {\beta},\bm b, \bm r} \left\{ \frac{1}{n}\sum\limits_{i = 1}^n {\rho_\tau({\boldsymbol{r}_i})}  + {\lambda _1}{\left\| \boldsymbol \beta \right\|_1} + {\lambda _2}{\left\| \boldsymbol b \right\|_1} \right\}\\
\text{s.t.} \  {\bm X}\bm \beta + \bm \gamma \beta_0 + \boldsymbol{r} =\bar{\bm y}, \boldsymbol{F\beta}=\bm{b}.
\end{array}
\end{equation}
The augmented Lagrangian of (\ref{qflasso_con1}) is
\begin{align}\label{auadmm}
&\tilde{ L}(\beta_0,\bm \beta, \bm b, \bm r, \bm u,\bm v) = \frac{1}{n}\sum\limits_{i = 1}^n {\rho_\tau({\boldsymbol{r}_i})} + {\lambda _1}{\left\| \boldsymbol \beta \right\|_1} + {\lambda _2}{\left\| \boldsymbol b \right\|_1} -  \boldsymbol{u}^\top({\bm X}\bm \beta + \bm \gamma \beta_0 + \boldsymbol{r}-\bar{\bm y})  \notag \\
&
 -  \bm{v}^\top (\bm {F\beta} -\bm  b)  + \frac{{{\mu _1}}}{2}\left\|  {\bm X}\bm \beta + \bm \gamma \beta_0 + \boldsymbol{r}-\bar{\bm y} \right\|_2^2 + \frac{{{\mu _2}}}{2}\left\| {\boldsymbol{F\beta}  - \bm{b}} \right\|_2^2 ,
\end{align}
where $\boldsymbol u \in \mathbb{R}^{n}$, $\boldsymbol v \in \mathbb{R}^{p-1}$ are the dual variables corresponding to the linear constraints ${\bm X}\bm \beta + \bm \gamma \beta_0 + \boldsymbol{r} =\bar{\bm y}$ and $\boldsymbol{F\beta=b}$, respectively. Here, the quadratic terms $\frac{{{\mu _1}}}{2}\left\|  {\bm X}\bm \beta + \bm \gamma \beta_0 + \boldsymbol{r}-\bar{\bm y} \right\|_2^2$ and $\frac{{{\mu _2}}}{2}\left\| {\boldsymbol{F\beta  - b}} \right\|_2^2$ penalize the violation of two linear
constraints; two given parameters $\mu_1$ and  $\mu_2$  are  always greater than $0$. 

We aim to find a saddle point $(\beta_0^*, \bm \beta^*, \bm b^*, \bm r^*, \bm u^*, \bm v^*)$ of the augmented Lagrangian function $\tilde{L}(\beta_0, \bm \beta,  \bm b, \bm r, \bm u, \bm v)$ satisfying the conditions:
\begin{equation}\label{saddle point}
\tilde{L}(\beta_0^*, \bm \beta^*, \bm b^*, \bm r^*, \bm u, \bm v) \le \tilde{L}( \beta_0^*, \bm \beta^*,\bm b^*, \bm r^*, \bm u^*, \bm v^*) \le \tilde{L}(\beta_0, \bm \beta, \bm b, \bm r, \bm u^*, \bm v^*)
\end{equation}
for all primal and dual variables $\beta_0$, $\bm \beta$, $\bm b$, $\bm r$, $\bm u$, and $\bm v$. It has been proven that $(\beta_0^*, \bm \beta^*)$ is an optimal solution of (\ref{qflasso_con1}) if and only if $(\beta_0^*, \bm \beta^*, \bm b^*, \bm r^*, \bm u^*, \bm v^*)$ can solve the above saddle point problem for some $\bm b^*$, $\bm r^*$, $\bm u^*$, and $\bm v^*$.
Hence, we can solve the saddle point problem using an iterative algorithm by alternating between primal and dual optimization steps, as follows:
\begin{equation}\label{itealgorithm}
\left\{ \begin{split}
\text{Primal}:\ &({\beta_0 ^{k + 1}},{\bm \beta^{k + 1}},{\bm b^{k + 1}},\bm r^{k+1}) = \mathop {\text{argmin}}\limits_{\beta_0,\bm \beta,\bm b,\bm r} \tilde{L}(\beta_0,\bm \beta, \bm b, \bm r,{\bm u^k},{\bm v^k}); \\
\text{Dual}:\ &{\bm u^{k + 1}} = \bm u^k - {\mu _1}( {\bm X}\bm \beta^{k+1} + \bm \gamma \beta_0^{k+1} + \bm r^{k+1}-\bar{\bm y}),\\
&\bm v^{k + 1} = \bm v^k - {\mu _2}(\bm F{\bm \beta ^{k + 1} - \bm b^{k + 1}}),
\end{split} \right.
\end{equation}
Here, the primal optimization step updates the four primal variables based on the current estimates of the two dual variables, while the dual optimization step updates the two dual variables based on the current estimates of the four primal variables.
Since the augmented Lagrangian function $\tilde{ L}$  is linear in $\boldsymbol{u}$ and $\boldsymbol{v}$, then the update of the dual variables is quite easy. Thus, the efficiency of the iterative algorithm (\ref{itealgorithm}) hinges entirely on the speed at which the primal problem can be solved. The faster we can obtain the optimal values for $\beta_0$, $\bm \beta$, $\bm b$, and $\bm r$ in the primal step, the more efficient the algorithm will be.

The augmented Lagrangian function $\tilde{ L}$ in (\ref{auadmm}) still has nondifferentiable terms, however, different from the original objective function in (\ref{unified}), the $\ell_1$-norm in  (\ref{qflasso_con1}) no longer contains inseparable regularization term. As a result, we can solve the primal optimization problem by alternating the minimization of $\beta_0$, $\bm \beta$, $\bm b$ and $\bm r$.
After organizing the formula in $\tilde{L}$ and excluding some constant terms, the update of primal variables in the iterative process can be expressed as follows,
\begin{small}
\begin{equation}\label{primalupdate}
\left\{ \begin{array}{l}
(\beta_0^{k+1},\bm \beta^{k+1}) = \underset{\beta_0,\bm \beta}{\arg \min}  {\lambda _1}{\| \boldsymbol{\beta} \|_1} + \frac{{{\mu _1}}}{2}\| {\bm X}\bm \beta + \bm \gamma \beta_0^k + \boldsymbol{r}^k-\bar{\bm y}-\frac{\bm u^k}{\mu_1}\|_2^2 + \frac{{{\mu _2}}}{2}\| {\boldsymbol{F\beta} - {\boldsymbol{b}^k}}-\frac{\bm v^k}{\mu_2} \|_2^2 \\ %
{\boldsymbol{b}^{k + 1}} = \underset{\boldsymbol{b}}{\arg \min}  {\lambda _2}{\| \boldsymbol{b} \|_1}  + \frac{{{\mu _2}}}{2}\| {\boldsymbol{F\beta}^{k+1} - {\boldsymbol{b}}}-\frac{\bm v^k}{\mu_2} \|_2^2,  \\ %
{\boldsymbol{r}^{k + 1}} = \underset{\boldsymbol{r}}{\arg \min} \frac{1}{n}\sum\limits_{i = 1}^n {{\rho _\tau }({\boldsymbol{r}_i})}  + \frac{{{\mu _1}}}{2}\|  {\bm X}\bm \beta^{k+1} + \bm \gamma \beta_0^{k+1} + \boldsymbol{r}-\bar{\bm y}-\frac{\bm u^k}{\mu_1}\|_2^2. 
\end{array} \right.
\end{equation}
\end{small}
According to \cite{W2023,LMY}, prior to solving each subproblem in (\ref{primalupdate}), we introduce two operators that help us effectively express the closed-form solutions for $\beta_0,\bm \beta, \bm b$, and $\bm r$.

$\bullet$ The minimization problem 
$$\mathop {\arg \min }\limits_{\boldsymbol{z}} \lambda {\left\| \boldsymbol{z} \right\|_1} + \frac{\mu }{2}\left\| {\boldsymbol{z} - \boldsymbol{z}^0} \right\|_2^2$$
with $\mu > 0$ and a given vector $\boldsymbol{z}^0$, has a closed-form solution defined as
\begin{equation}\label{opera1}
\boldsymbol{z}_i^* = \text{sign}({\boldsymbol{z}^0_i}) \cdot \text{max}\{ 0,\left| {{\boldsymbol{z}^0_i}} \right| - {\lambda \mathord{\left/
 {\vphantom {1 \mu }} \right.
 \kern-\nulldelimiterspace} \mu }\},
\end{equation}
where $\boldsymbol{z}_i^*$ and $\boldsymbol{z}^0_i$ are the $i$-th components of vectors $\bm z^*$ and $\bm z^0$, respectively.

$\bullet$ The minimization problem 
$$\mathop {\arg \min }\limits_{\boldsymbol{z}} \frac{1}{n} \sum\limits_{i = 1}^n {{\rho _\tau }({\boldsymbol{z}_i})}  + \frac{\mu }{2}\left\| {\boldsymbol{z} - \boldsymbol{z}^0} \right\|_2^2$$
with $\mu > 0$ and a given vector $\boldsymbol{z}^0$, has a closed-form solution defined as
\begin{equation}\label{opera2}
\boldsymbol{z}_i^* = \max \{ {\boldsymbol{z}^0_i} - {\tau  \mathord{\left/
 {\vphantom {\tau  {(n}}} \right.
 \kern-\nulldelimiterspace} {(n}}\mu ),\min (0,{\boldsymbol{z}^0_i} + {{(1 - \tau )} \mathord{\left/
 {\vphantom {{(1 - \tau )} {(n\mu) }}} \right.
 \kern-\nulldelimiterspace} {(n\mu) }})\},
\end{equation}
where $\boldsymbol{z}_i^*$ and $\boldsymbol{z}^0_i$ are the $i$-th components of vectors $\bm z^*$ and $\bm z^0$, respectively.

Now, we discuss the update of each subproblem in detail. First, for the subproblem of updating $\beta_0$ and $\boldsymbol{\beta}$ in (\ref{primalupdate}), the objective function on minimizing $\bm \beta$  can be converted into a problem similar to Lasso.
More specifically,  take $\tilde \beta= (\beta_0, \bm \beta^\top)^\top$, ${{\tilde {\boldsymbol{y}}}^k} = {\left(\sqrt {{\mu _1}} {(\bar{\boldsymbol{y}} - {\boldsymbol{r}^k} + \boldsymbol{u}^k/{\mu_1})^\top},\sqrt {{\mu _2}} {({\boldsymbol{b}^k} + \boldsymbol{v}^k/{\mu_2})^\top}\right)^\top}$ and $\tilde {\boldsymbol{X}} = {(\sqrt {{\mu _1}} \bar{\boldsymbol{X}}^\top,\sqrt {{\mu _2}} \bar{\boldsymbol{F}}^\top)^\top}$, where $\bar{\bm F}=[\bm 0_{p-1},\bm F]$. Then we have
\begin{equation}\label{ladmm_beta}
{\tilde {\boldsymbol{\beta}} ^{k + 1}} = \mathop {\arg \min }\limits_{\tilde{\bm \beta}}  \lambda_1 {\| \boldsymbol{\beta} \|_1} + \frac{1}{2}\| {\tilde{\boldsymbol{X}}\tilde{\boldsymbol{\beta}}  - {{\tilde {\boldsymbol{y}}}^k}} \|_2^2.
\end{equation}
Obviously, formula (\ref{ladmm_beta}) does not have a closed-form solution due to the non-identity matrix $\tilde{\boldsymbol{X}}$.
The existing literatures generally have two methods to achieve the iteration of this step. One is to use the coordinate descent algorithm to solve it, such as scdADMM in \cite{GFK}. The other is to linearize the quadratic term in (\ref{ladmm_beta})  to get an approximate solution, such as pADMM in \cite{GFK} and LADMM in \cite{LMY}. Although coordinate descent algorithm can solve this Lasso subproblem quickly, it needs to complete an inner iteration for each iteration of ADMM. As a result, when $p$ is particularly large, this method may have a great computational consumption. 
Therefore, this paper focuses on the linearization method. 

Similar to  \cite{LMY,L2023}, linearize the quadratic term $\|{\tilde{\boldsymbol{X}}\tilde{\boldsymbol{\beta}}  - {{\tilde {\boldsymbol{y}}}^k}} \|_2^2/2$ and replace it by
$${\left[{\tilde{\boldsymbol{X}}^\top}(\tilde {\boldsymbol{X}}{\boldsymbol{\tilde \beta} ^k} - {\tilde{\boldsymbol{y}}^k})\right]^\top}(\boldsymbol{\tilde \beta}  - {\boldsymbol{\tilde \beta} ^k}) + {{(\eta } \mathord{\left/
 {\vphantom {{(\eta } 2}} \right.
 \kern-\nulldelimiterspace} 2})\| {\boldsymbol{\tilde \beta}  - {\boldsymbol{\tilde \beta} ^k}}\|_2^2,$$
where the linearized  parameter $\eta > 0$ controls the proximity to $\tilde{\boldsymbol{\beta}}^k$. To ensure the convergence of the algorithm, $\eta$ need to be larger than $\rho(\tilde{{\boldsymbol{X}}}^\top \tilde{\boldsymbol{X}})$, where $\rho(\tilde{{\boldsymbol{X}}}^\top \tilde{\boldsymbol{X}})$ represents the maximum eigenvalue of $\tilde{{\boldsymbol{X}}}^\top \tilde{\boldsymbol{X}}$. In \cite{L2023}, they suggest using the power method \cite{GL} to compute the maximum eigenvalue of a positive definite high-dimensional matrix within the framework of linearized ADMM. We adopt this efficient method to calculate $\eta$ in our work as well.
Thus, we can solve the following problem
\begin{equation}\label{ladmm_beta1}
{\boldsymbol{\tilde \beta} ^{k + 1}} = \mathop {\arg \min }\limits_{\boldsymbol{\tilde \beta}}\left\{  \lambda_1 {\| \boldsymbol{\beta} \|_1} + (\eta / 2)\| \tilde {\boldsymbol{\beta}}  - \tilde {\boldsymbol{\beta}} ^k + \tilde{\boldsymbol{X}}^\top(\tilde{\boldsymbol{X}} \tilde{\boldsymbol{ \beta}}^k-\tilde{\boldsymbol{y}}^k)/{\eta} \|_2^2 \right\}
\end{equation}
to get an approximate solution of the $\boldsymbol{\tilde \beta}$-subproblem in (\ref{ladmm_beta}). Then,
\begin{equation}\label{admmu_beta0}
\beta_0^{k+1}={\boldsymbol{\tilde \beta} ^{k}}_1-({\tilde{\boldsymbol{X}}^\top(\tilde{\boldsymbol{X}}{\boldsymbol{\tilde \beta} ^k} - {{\tilde {\boldsymbol{y}}}^k})})_1 / {\eta }.
\end{equation}
And  it follows from (\ref{opera1}) that the closed-form solution of updating $\bm \beta$ is obtained component-wisely by 
\begin{equation}\label{ladmm_beta2}
\boldsymbol{\beta} _i^{k + 1} = \text{sign}\left(\boldsymbol{\beta} _i^k - ({\tilde{\boldsymbol{X}}^\top(\tilde {\boldsymbol{X}}{\tilde{\boldsymbol{\beta}} ^k} - {{\tilde {\boldsymbol{y}}}^k})})_{i+1} / {\eta }\right) \cdot \text{max}\left\{0, |\boldsymbol{\beta}_i^k-(\tilde{\boldsymbol{X}}^{\top}(\tilde{\boldsymbol{X}}\tilde{\boldsymbol{\beta}}^k-\boldsymbol{\tilde{y}}^k))_{i+1}/{\eta}|-\lambda_1/{\eta}
\right\}.
\end{equation}
Here,  $( \  )_i$ represents the $i$-th element of the vector in parentheses. 

Second, for the subproblem of updating $\boldsymbol{b}$ in (\ref{primalupdate}), its closed-form solution which follows from (\ref{opera1}) can be got component-wisely by 
\begin{equation}\label{admmu_b}
 \boldsymbol{b}_i^{k + 1} = \text{sign}\left({( \boldsymbol{F}{ \boldsymbol{\beta} ^{k + 1}})_i} -  \boldsymbol{v}_i^k /\mu_2\right) \cdot \text{max}\left\{ 0,|(\boldsymbol{F\beta} ^{k + 1})_i - \boldsymbol{v}_i^k /\mu_2| - \lambda_2/\mu_2\right\}.
\end{equation}

Finally, for the subproblem of  updating $\boldsymbol{r}$ in (\ref{primalupdate}), its closed-form solution which follows from (\ref{opera2}) can also be given component-wisely by 
\begin{small}
\begin{equation}\label{admmu_r}
\boldsymbol{r}_i^{k + 1} = \max \left\{ {\bar{\boldsymbol{y}}_i} - (\boldsymbol{X\beta} ^{k + 1})_i -\beta_0^{k+1} \bm \gamma_i + \frac{{\boldsymbol{u}_i^k}}{{{\mu _1}}} - \frac{\tau }{{n{\mu _1}}},\min\left(0,{\bar{\boldsymbol{y}}_i} - (\boldsymbol{X\beta} ^{k + 1})_i - \beta_0^{k+1} \bm \gamma_i + \frac{{\boldsymbol{u}_i^k}}{{{\mu _1}}} + \frac{{(1 - \tau )}}{{n{\mu _1}}}\right)
\right\}.
\end{equation}
\end{small}
To summarize, the iteration of MLADMM for the unified form (\ref{unified}) can be described in Algorithm \ref{alg1}.
\begin{algorithm}
\caption{MLADMM for solving the unified form (\ref{unified})}
\label{alg1}
\begin{algorithmic}
\STATE {\textbf{Input:} Observation data: $\boldsymbol{X},\boldsymbol{y}$; primal variables: $\beta_0^0, \boldsymbol{\beta}^0,\boldsymbol{b}^0,\boldsymbol{r}^0$; dual variables: $\boldsymbol{u}^0,\boldsymbol{v}^0$; augmented parameters: $\mu_1, \mu_2$; penalty parameter: $\lambda_1, \lambda_2$; and quantile level $\tau  \in [0,1]$.}
\STATE {\textbf{Output:} the total number of iterations $K$, $\beta_0^K$, $\boldsymbol{\beta}^K,\bm b^K, \bm r^K, \bm u^K, \bm v^K$. }
\STATE {\textbf{while} not converged \textbf{do}}
\STATE {\quad 1. Update ${\beta}_0^{k+1}$ using (\ref{admmu_beta0})},
\STATE {\quad 2. Update $\boldsymbol{\beta}^{k+1}$ using (\ref{ladmm_beta2})},
\STATE {\quad 3. Update $\boldsymbol{b}^{k+1}$ using (\ref{admmu_b})},
\STATE {\quad 4. Update $\boldsymbol{r}^{k+1}$ using (\ref{admmu_r})},
\STATE {\quad 5. Update $\boldsymbol{u}^{k+1}$ and $\boldsymbol{v}^{k+1}$ using (\ref{itealgorithm})},
\STATE {\textbf{end while}}
\STATE {\textbf{return} solution}.
\end{algorithmic}
\end{algorithm}

\subsection{Convergence and Computational Cost Analysis}
Clearly, Algorithm \ref{alg1} is not a traditional two-block linearized ADMM algorithm, but rather a four-block ADMM algorithm where the first variable is linearized.
Recently, Chen et al. \cite{CBYY} demonstrated that the direct extension of ADMMs for convex optimization with three or more separable blocks may not necessarily converge and provided an example of divergence. This finding has significantly hindered the use of multi-block ADMM in high-dimensional statistical learning. Thankfully, they also established a sufficient condition that ensures the convergence of the direct extension of multi-block ADMM. An important result of their study is that the convergence of multi-block ADMM is guaranteed when specific coefficient matrices are orthogonal, enabling the iteration of all primal subproblems to be sequentially divided into two independent parts.

 Following the approach proposed in \cite{W2023}, this optimization form can also be extended to solve other existing SFL models, including SFL-SR \cite{JLD}, SFL-LAD \cite{LTZ}, SFL-LS, and SFL-SVM \cite{TSRZ, YX}. In addition, we extended MLADMM to the SFL-SVR (support vector  regression \cite{YGS}) method, which plays an important role in global sensitivity analysis \cite{CLZ,ZLZ} and  reliability analysis \cite{GRC}. For more detailed information, please refer to supplementary materials \ref{B}.

Recall (\ref{primalupdate}),  we update  $\beta_0^{k+1}$ and ${\boldsymbol \beta}^{k+1}$ through $\boldsymbol b^{k}$ and $\boldsymbol r^k$. On the other hand, the updates of $\boldsymbol b^{k+1}$ and $\boldsymbol r^{k+1}$ are independent of each other, and   only related to $\beta_0^{k+1}$ and ${\boldsymbol \beta}^{k+1}$. Then, we can see that the iterations of primal variables subproblems can be sequentially divided into two independent parts, $\{\beta_0^{k+1},\bm \beta^{k+1}\}$ and $\{\bm b^{k+1}, \bm r^{k+1}\}$. As a result,  Algorithm \ref{alg1} can satisfy the convergence condition of multi-block ADMMs proposed by \cite{CBYY}. In addition, we give a more formal analysis about these in supplementary materials \ref{A}. Based on this, we discuss the convergence and computational costs of the proposed LADMM algorithm. 
The convergence property and  convergence rate of Algorithm \ref{alg1} are shown in the following theorem, which is proven in the supplementary materials \ref{A}.
\vspace{-1em}
\begin{thm}\label{TH1}
The sequence $\boldsymbol{w}^{k}= (\beta_0^k,\boldsymbol{\beta}^k, \boldsymbol{b}^k, \boldsymbol{r}^k, \boldsymbol{u}^k, \boldsymbol{v}^k)$ is generated by Algorithm \ref{alg1} with an initial feasible solution $\boldsymbol{w}^{0}= (\beta_0^0,\boldsymbol{\beta}^0, \boldsymbol{b}^0, \boldsymbol{r}^0, \boldsymbol{u}^0,{\boldsymbol{v}^0})$. The sequence $\boldsymbol{w}^{k}$ converges to 
$\boldsymbol{w}^{*}=(\beta_0^*,\boldsymbol{\beta}^{*},\boldsymbol{b}^{*}, \boldsymbol{r}^{*},
\boldsymbol{u}^{*}, \boldsymbol{v}^{*})$, where  $\boldsymbol{w}^{*}$ is an optimal solution point of the (\ref{qflasso_con1}). The $O(1/k)$ convergence rate in a non-ergodic sense can also be obtained, i.e.,
\begin{equation}
\| \boldsymbol{w}^{k}-\boldsymbol{w}^{k+1} \|_{\boldsymbol{H}}^{2} \le \small{\frac{1}{k+1}}\| \boldsymbol{w}^{0}-\boldsymbol{w}^{*} \|_{\boldsymbol{H}}^{2},
\end{equation}
where  $\boldsymbol{H}$ is a symmetric and positive semidefinite matrix  (we denote it $\|\boldsymbol w\|_{\boldsymbol H}:=\sqrt{\boldsymbol w^\top \boldsymbol H \boldsymbol w})$.
\end{thm}

Note that $\| \boldsymbol{w}^{0}-\boldsymbol{w}^{*} \|_{\boldsymbol{H}}^{2}$ is the order of $O(1)$, and as a consequence $\| \boldsymbol{w}^{k}-\boldsymbol{w}^{k+1} \|_{\boldsymbol{H}}^{2}=O(1/k)$. For the specific form of the positive semidefinite matrix $\boldsymbol{H}$, see the proof of Theorem \ref{TH1} in  supplementary materials \ref{B}.
Now, let us analyze the computational cost of Algorithm \ref{alg1} when design matrix $\boldsymbol{X}$  is a general numerical matrix and an identity matrix, respectively. When the design matrix of a fused Lasso model is the identity matrix, it is often called ``Fused Lasso Signal Approximator" (FLSA); see \cite{YWL}.  FLSA has been widely used in many fields, and many statistical learning algorithms have been proposed to solve it; see \cite{YWL,CDS} and its references.

Before the MLADMM iterations, we need to use the power method described in \cite{W2023} to calculate $\rho(\tilde{\boldsymbol{X}}^\top \tilde{\boldsymbol{X}})$. The power method is an iterative algorithm that only requires inputting $\tilde{\boldsymbol{X}}$. Each iteration of the power method has a time complexity of $O(p^2)$. Since $\eta > \rho(\tilde{\boldsymbol{X}}^\top \tilde{\boldsymbol{X})}$ is an inequality, the convergence conditions of the power method can be set loosely, for example, as $10^{-2}$. Consequently, the total number of iterations for the power method is typically no more than a few dozen.
For the $(\beta_0, \bm \beta)$-update, we first form 
\begin{align}\label{up1}
({\tilde{\boldsymbol{X}}^\top(\tilde{\boldsymbol{X}}{\tilde {\boldsymbol{\beta}} ^k} - {{\tilde {\boldsymbol{y}}}^k})})=\left[\begin{array}{*{20}{c}}
\mu_1 \bm \gamma^\top ({\bm X}\bm \beta^k + \bm \gamma \beta_0^k-\bar{\bm y} + \bm r^k - \bm u^k/\mu_1)    \\
\mu_1 \bm X^\top ({\bm X}\bm \beta^k + \bm \gamma \beta_0^k-\bar{\bm y} + \bm r^k - \bm u^k/\mu_1)+\mu_2 \bm F^\top(\bm F \bm \beta^k-\bm b^k-\bm v^k/\mu_2)
\end{array} \right] 
\end{align}
at a cost of $O(p^2)$ flops. When $({\tilde{\boldsymbol{X}}^\top(\tilde{\boldsymbol{X}}{\tilde {\bm{ \beta}} ^k} - {{\tilde {\boldsymbol{y}}}^k})})$  has been calculated, updating $\beta_0^{k+1}$ and $\bm \beta^{k+1}$ costs $O(p)$ flops.  Clearly,  the costs of updating $\bm b^{k+1}$ and $\bm r^{k+1}$ are $O(p)$ and $O(np)$ flops, respectively.  The update  of dual variables $\bm u^{k+1}$ and $\bm v^{k+1}$ costs $O(np)+O(p)$. Since $O(p)+O(p)+O(np)+O(np)+O(p)=O(np)$, the overall cost of Algorithm \ref{alg1} is $O(p^2)+ O(np) \times K$ flops, where $K$ is the total number of  MLADMM iterations. Similarly, when $\bm X=\bm I_p$, the overall cost of Algorithm \ref{alg1} is $O(p) + O(p) \times K$ flops. We summarize the following propositions.
\begin{prop}\label{prop2}
For the case when $p>>n$,  the overall computational cost of Algorithm \ref{alg1} is
\begin{equation}\label{cost}
\left\{ \begin{split}
O(p^2)+O(np)  \times K, \   &\text{if} \ \boldsymbol{X} \in \mathbb{R}^{n\times p}, \\
O(p)+O(p) \times K, \  &\text{if} \ \boldsymbol{X}=\boldsymbol I_p,
\end{split} \right.
\end{equation}
where $K$  represents the  total number of iterations and $\boldsymbol I_p$ a $p$-dimensional identity matrix.
\vspace{-1em}
\end{prop}
\begin{rem}
In fact, the ultrahigh dimension setting $p>>n$ does not appear in FLSA  because there is no $n \times p$ design matrix $\bm X$ for approximating $n$-dimensional signals $\boldsymbol y$. In the case of FLSA ($\boldsymbol{X}=\boldsymbol I_p$), then $n=p$ and thus Algorithm \ref{alg1} for quantile FLSA is $O(n)\times K$ complexity. 
\vspace{-1em}
\end{rem}

\textcolor{red}{Proposition \ref{prop2} shows that Algorithm \ref{alg1} for SFL-QRE has the same complexity as two-block LADMM for solving ultrahigh dimensional SFL-LS regression in  \cite{LMY} and elastic-net SVM  classifications in  \cite{L2023}. In addition, multi-block ADMMs for sparse fused Lasso regression and classification in \cite{YX,W2023}  have  at least  $O(p^2) \times K$ complexity, which  increases faster than our algorithm if $p$ has a large increase.
In general, when employing ADMM algorithms to solve SFL regularized models, the algorithm complexity is approximately $O(p^2) \times K$. On the contrast, for linearized ADMM algorithms tackling SFL regularized models, the complexity is approximately $O(np) \times K$. In addition, for standard pinball-SVM in \cite{HSS}, the solving algorithms employed are quadratic programming-based \cite{HSS}, with a computational complexity of $O(n^3)$, or SMO algorithms \cite{HSS2}, with a computational complexity ranging from $O(n)$ to $O(n^2)$.}

\section{Numerical Results}\label{sec6}
In this section,  we compare the performance of the proposed MLADMM algorithm and the existing state-of-the-art algorithms in synthetic data and real data. In general, all the parameters (for MLADMM) $\mu_1, \mu_2, \lambda_1$ and $\lambda_2$ need to be chosen by the cross-validation (CV) method, while leading to the high computational burden. To reduce the computational cost,  we  suggest that $\mu=\mu_1=\mu_2$,  and $\mu$ is selected from the set $\{0.01,0.1,1\}$. Moreover, $\lambda_1, \lambda_2 \in (0,1]$ increases by $0.01$ every time from $0.01$ until $1$ stops. All the experiments in this paper prove that the selection of these parameters is very effective.  We put the R codes for implementing the proposed MLADMM and reproducing our experiments on  \url{https://github.com/xfwu1016/LADMM-for-qfLasso}.  The iteration initial values of the prime and dual variables are both set to be $\bm 0$, that is, $\beta_0^0=0, {\bm \beta}^0=\bm 0_{p+1}, \bm b^0=\bm 0_{p-1}, \bm r^0=\bm 0_n, \bm u^0=\bm 0_n$ and $\bm v^0= \bm 0_{p-1}$. 

In order to accelerate the convergence speed of MLADMM algorithm, we adopt the following two acceleration  measures  at the same time. One is the effective and simple method of adjusting $\mu$ proposed by \cite{BYW}, that is
\begin{align}\label{sat}
\mu^{k+1}= \left\{ \begin{array}{l}
\mu^k*c_2, \  \text{if} \ c_1 \times \text{primal residual} < \text{dual residual},\\
\mu^k/c_2,\ \ \ \text{if} \ \text{primal residual} > c_1 \times \text{dual residual} ,\\
\mu^k, \qquad \ \text{otherwise},
\end{array} \right.
\end{align}
where $c_1=10, c_2=2$, $\text{primal residual}^2=\|\mu \bm \gamma^\top(\bm r^k-\bm r^{k+1}) \|_2^2 + \|\mu \bm X^\top(\bm r^k-\bm r^{k+1}) +\mu \bm F^\top(\bm b^{k+1}-\bm b^k) \|_2^2$ and $\text{dual residual}^2= \|{\bm X}\bm \beta^{k+1} + \bm \gamma \beta_0^{k+1} + \boldsymbol{r}-\bar{\bm y}\|_2^2+\|\bm F \bm \beta^{k+1}-\bm b^{k+1} \|_2^2$.  This self-adaptive tuning method has been used by Boyd et al. \cite{SNEBJ}, and they also claimed that the convergence of ADMM can be guaranteed  if $\mu^k$ becomes fixed after a finite number of iterations. The other is we select $\eta > 0.75 \rho(\tilde{{\boldsymbol{X}}}^\top \tilde{\boldsymbol{X}})$, which has been proved by He et al. \cite{BFX}  that it can accelerate the convergence rate of LADMM.   Theoretically, as long as $\eta > \rho(\tilde{{\boldsymbol{X}}}^\top \tilde{\boldsymbol{X}})$, the convergence of the algorithm can be ensured. However, the larger the value of $\eta$, the slower the convergence speed of the algorithm. Therefore, it is meaningful to reduce the $\eta$ value to improve the efficiency of the algorithm.

The proposed MLADMM algorithm is iterated until some stopping criterion is satisfied. We use the stopping criterion from the Section 3.3.1 of \cite{SNEBJ}. To be specific, the MLADMM  algorithm is terminated either when the iterative sequence  $\{(\beta_0^{k+1},{\bm \beta}^{k+1}, \bm b^{k+1}, \bm r^{k+1}, \bm u^{k+1}, \bm v^{k+1} )\}$ satisfies the following criterion
\begin{gather}
\text{primal residual} \le \sqrt{p+1}\epsilon_1 +  \epsilon_2 \sqrt{\| \bm \gamma^\top \bm u^{k+1} \|_2^2+\|\bm X^\top\bm u^{k+1}+\bm F^\top \bm v^{k+1} \|_2^2} , \notag \\
 \text{dual residual}  \le \sqrt{n+p-1} \epsilon_1 + \epsilon_2 \sqrt{\max\left\{\|\bm X \bm \beta^{k+1}+\bm \gamma \beta_0^{k+1}\|_2^2 + \|\bm F \bm \beta^{k+1} \|_2^2, \| \bm b^{k+1} \|_2^2 + \| \bm r^{k+1} \|_2^2,  {n}\right\}},  \notag
\end{gather}
where $\epsilon_1=10^{-4}$ and $\epsilon_2=10^{-4}$, or when the number of MLADMM iterations exceeds a given number, such as $500$. All numerical experiments were run on R (version 4.1.1) software on the Inter E5-2650 2.0 GHz processor with 16 GB memory.

\subsection{Synthetic Data}\label{sec61}
\subsubsection{Classification}\label{sec41}
$\bullet$  \textbf{Example 1}:  We use the two-dimensional example which was used in \cite{HSS}. The two classes ``$+$" and ``$-$" are generated from the distributions $N(\bm \mu_+,\Sigma_+)$ and $N(\bm \mu_-,\Sigma_-)$, where $\bm \mu_+=(0.5,-3)^\top$, $\bm \mu_-=(-0.5,3)^\top$ and
$\Sigma_+=\Sigma_-=\left[\begin{array}{*{20}{c}}
0.2 & 0\\
0 & 3
\end{array} \right] 
$. In this example, the Bayes classifier is $F=2.5 x_1-x_2$. We set $n =50,100,200$ and $500$, and use the pinball-SVM in \cite{HSS}, the SFL-SVM in  \cite{YX} and the proposed SFL-PSVM to estimate the classification boundary $x_2=\beta_1 x_1+\beta_0$. The classification boundary derived by the Bayes
classifier is $x_2=2.5 x_1$. Then, the ideal result is $\beta_1=2.5$ and $\beta_0=0$. We repeat each experiment 100 times in Table  \ref{Tab1} and \textcolor{red}{the left half of Figure \ref{Fig8}}, and  record the mean and standard deviation in Table  \ref{Tab1}.

\textcolor{red}{As suggested by a reviewer, it is important to study convergence analysis of our algorithm before comparing results. In the left half of Figure \ref{Fig8}, we compared the variation of primal residuals (since the dual residuals converge faster, only the original residuals are presented) with increasing iteration times for different sample sizes. The findings suggest that, overall, MLADMM exhibits rapid convergence. However, as the sample size grows, MLADMM tends to converge at a slower pace. }
\begin{table}[h]\tiny
  \caption{\small{Classification boundary for Example 1.}}
  \centering
  \renewcommand{\arraystretch}{1.3}
  \resizebox{\textwidth}{!}{
  \begin{tabular}{ccccccc}
  \Xhline{0.8pt}
                                &  Methods                  &           &  $n=50$         &  $n=100$        &  $n=200$        &  $n=500$  \\
  \Xhline{0.5pt}
  \multirowcell{4}{$\tau=1$}    & \multirowcell{2}{pinball-SVM} & $\beta_1$ & $2.469\pm1.016$ & $2.547\pm0.671$ & $2.534\pm0.527$ & $2.586\pm0.272$\\
                                &                           & $\beta_0$ & $0.004\pm0.372$ & $-0.006\pm0.263$ & $0.002\pm0.182$ & $-0.001\pm0.098$\\
                                & \multirowcell{2}{SFL-PSVM} & $\beta_1$ & $2.497\pm0.721$ & $2.488\pm0.389$ & $2.445\pm0.362$ & $2.424\pm0.201$\\
                                &                           & $\beta_0$ & $0.009\pm0.254$ & $0.002\pm0.174$ & $0.003\pm0.152$ & $0.002\pm0.123$\\
  \multirowcell{4}{$\tau=0.5$}  & \multirowcell{2}{pinball-SVM} & $\beta_1$ & $2.451\pm1.213$ & $2.561\pm0.753$ & $2.565\pm0.610$ & $2.572\pm0.263$\\
                                &                           & $\beta_0$ & $0.005\pm0.315$ & $0.024\pm0.216$ & $-0.006\pm0.203$ & $0.008\pm0.121$\\
                                & \multirowcell{2}{SFL-PSVM} & $\beta_1$ & $2.623\pm0.697$ & $2.487\pm0.412$ & $2.538\pm0.427$ & $2.605\pm0.231$\\
                                &                           & $\beta_0$ & $0.089\pm0.195$ & $0.030\pm0.152$ & $0.045\pm0.143$ & $0.006\pm0.140$\\   
\multirowcell{4}{$\tau=0.1$}  & \multirowcell{2}{pinball-SVM} & $\beta_1$ & $2.697\pm1.124$ & $2.632\pm0.763$ & $2.593\pm0.585$ & $2.538\pm0.289$\\
                                &                           & $\beta_0$ & $0.034\pm0.589$ & $-0.027\pm0.203$ & $-0.007\pm0.215$ & $0.004\pm0.133$\\
                                & \multirowcell{2}{SFL-PSVM} & $\beta_1$ & $2.531\pm0.530$ & $2.383\pm0.361$ & $2.449\pm0.273$ & $2.459\pm0.203$\\
                                &                           & $\beta_0$ & $0.060\pm0.217$ & $0.039\pm0.149$ & $0.006\pm0.156$ & $0.003\pm0.101$\\
                                & \multirowcell{2}{SFL-SVM}  & $\beta_1$ & $2.362\pm0.831$ & $2.497\pm0.521$ & $2.542\pm0.471$ & $2.413\pm0.212$\\
                                &                           & $\beta_0$ & $0.017\pm0.367$ & $0.046\pm0.203$ & $0.015\pm0.224$ & $0.019\pm0.198$\\
  \Xhline{0.8pt}
  \end{tabular}}
\label{Tab1}
\end{table}

 The above three SVMs  converge to the Bayes classifier and the mean values are pretty
good. However, the  standard deviation of the SFL-SVM is
obviously larger than that of the SFL-PSVM. This observation implies that the SFL-PSVM is more stable than SFL-SVM for resampling,
which shows the SFL-PSVM has the potential advantage in ultrahigh-dimensional problems.
\begin{figure}[h]
  \centering
  \includegraphics[width=16cm]{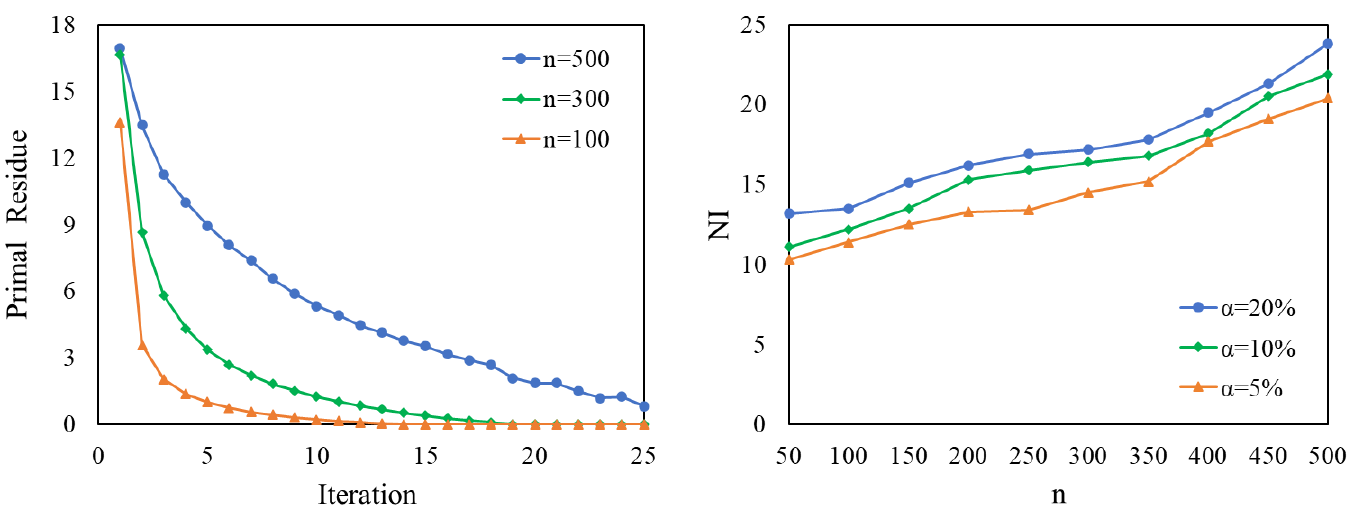}
  \caption{\small{Left figure: The variation of the primal residual with increasing number of iterations. Right figure: The total number of iterations (NI) varies with the increase of samples.}}
  \label{Fig8}
\end{figure}

$\bullet$  \textbf{Example 2}:  Like \cite{HSS}, we add noise to further demonstrate the stability of SFL-PSVM resampling. The labels of the noise points are chosen from $\{ +, - \} $ with equal
probability. The positions of these noised points are generated from Gaussian distribution $N(\bm \mu_n,\Sigma_n)$ with $\bm \mu_n=(0,0)^\top$ and $\Sigma_n=\left[\begin{array}{*{20}{c}}
1 & -0.8\\
-0.8 & 1
\end{array} \right] 
$.  These noised points affect the labels around the boundary and the level of the noise is controlled by the ratio of the noise data in the training set, denoted by $\alpha$. Note that the noise does not change the Bayes classifier, but it can affect the result of SFL-SVM. We give the mean and the standard deviation for repeating  the sampling and training process 100 times in \textcolor{red}{the right half of Figure \ref{Fig8}} and  Table \ref{Tab2}.

\begin{table}[h]\tiny
  \caption{\small{Classification Boundary for Noise Corrupted Data in Example 2.}}
  \centering
  \renewcommand{\arraystretch}{1.3}
  \resizebox{\textwidth}{!}{
  \begin{tabular}{ccccccc}
  \Xhline{0.8pt}
                                &  Method                  &           & $n=100,\textcolor{red}{\alpha= 5 \%}$ & $n=100,\textcolor{red}{\alpha= 10\%}$ & $n=200,\textcolor{red}{\alpha=5\%}$ & $n=200,\textcolor{red}{\alpha=10\%}$  \\
  \Xhline{0.5pt}
  \multirowcell{4}{$\tau=1$}    & \multirowcell{2}{pinball-SVM} & $\beta_1$ & $2.403\pm0.796$ & $2.365\pm0.815$ & $2.349\pm0.520$ & $2.331\pm0.565$\\
                                &                           & $\beta_0$ & $-0.065\pm0.258$ & $0.047\pm0.273$ & $0.042\pm0.206$ & $0.017\pm0.172$\\
                                & \multirowcell{2}{SFL-PSVM} & $\beta_1$ & $2.548\pm0.562$ & $2.641\pm0.631$ & $2.532\pm0.473$ & $2.503\pm0.472$\\
                                &                           & $\beta_0$ & $0.053\pm0.241$ & $0.052\pm0.249$ & $0.019\pm0.173$ & $0.021\pm0.163$\\
  \multirowcell{4}{$\tau=0.5$}  & \multirowcell{2}{pinball-SVM} & $\beta_1$ & $2.317\pm0.723$ & $2.257\pm0.768$ & $2.359\pm0.539$ & $2.278\pm0.534$\\
                                &                           & $\beta_0$ & $-0.038\pm0.230$ & $-0.014\pm0.282$ & $0.044\pm0.172$ & $-0.015\pm0.194$\\
                                & \multirowcell{2}{SFL-PSVM} & $\beta_1$ & $2.445\pm0.543$ & $2.419\pm0.587$ & $2.455\pm0.454$ & $2.491\pm0.454$\\
                                &                           & $\beta_0$ & $0.047\pm0.219$ & $0.019\pm0.265$ & $0.029\pm0.153$ & $0.028\pm0.171$\\          
 \multirowcell{4}{$\tau=0.1$}  & \multirowcell{2}{pinball-SVM} & $\beta_1$ & $2.197\pm0.865$ & $2.116\pm0.523$ & $2.136\pm0.517$ & $2.145\pm0.498$\\
                                &                           & $\beta_0$ & $-0.034\pm0.242$ & $0.032\pm0.214$ & $0.032\pm0.205$ & $0.019\pm0.215$\\
                                & \multirowcell{2}{SFL-PSVM} & $\beta_1$ & $2.591\pm0.521$ & $2.591\pm0.416$ & $2.395\pm0.443$ & $2.395\pm0.437$\\
                                &                           & $\beta_0$ & $0.037\pm0.199$ & $0.015\pm0.271$ & $0.034\pm0.175$ & $0.029\pm0.184$\\
                                & \multirowcell{2}{SFL-SVM}  & $\beta_1$ & $1.424\pm0.536$ & $1.334\pm0.679$ & $1.655\pm0.515$ & $1.597\pm0.513$\\
                                &                           & $\beta_0$ & $0.053\pm0.227$ & $0.026\pm0.294$ & $0.042\pm0.203$ & $0.029\pm0.329$\\
  \Xhline{0.8pt}
  \end{tabular}}
\label{Tab2}
\end{table}
\textcolor{red}{Before analyzing and comparing the results, we will investigate the convergence of our algorithm in the presence of noise in the right half of Figure \ref{Fig8}. Unlike the left half, which depicts the variation of the primal residual with increasing iteration times, the right half reflects the total number of iterations required by the algorithm for each sample under different noise levels. This result indicates that it is consistent with the conclusion shown in the left half of the graph, as the number of samples increases, the number of iterations required by the algorithm will slightly increase. Although our method may increase with the increase of noise ratio, the increase is not significant.} In this example, the classification boundary calculated by the SFL-SVM is quite different from the Bayes classifier, confirming the sensitivity of the SFL-SVM to noise around the boundary. In the next example, we will show the performance of these SVMs on ultrahigh dimensional datasets.  \textcolor{red}{For additional experiments on other noises such as $t$, Cauchy, mixed normal, lognormal, and more estimation metrics, we have included them in Table \ref{tab7} in Supplementary Material \ref{C}.}

$\bullet$  \textbf{Example 3}:  The ultrahigh dimensinal classification data are generated from the distributions $N(\bm \mu_+,\Sigma)$ and $N(\bm \mu_-,\Sigma)$, respectively, where $\bm \mu_+=(1,\dots,1,0,\dots,0)^\top$ with the first 10 elements being 1 and the rest being 0,  $\bm \mu_-=(-1,\dots,-1,0,\dots,0)^\top$ with the first 10 elements being -1 and the rest being 0, and the covariance matrix is
\begin{equation}
\boldsymbol{\Sigma}=\left[\begin{array}{*{20}{c}}
\bm {\Sigma}_{10}^*&  \bm0 \\ 
\bm 0 &\bm I_{(p-10)}
\end{array} \right].
\end{equation}
Here, $\bm {\Sigma}_{10}^*$ is a $10 \times 10$ matrix with all diagonal elements being 1 and the rest being $\rho<1$. Note that  the larger $\rho$ implies the higher correlation between the first 10 features of the input data. This data has been used in \cite{L2023}  for SVMs. Here, we  implement our experiments on $\rho=0.5$.  Clearly,  the first 10 features play a key role on the classification  boundary. For this example, the Bayes classifier is $F=\sum_{j=1}^{10}x_j$.   We also add noise into this training set as in Example 2 with $\bm \mu_n=(0,0, \dots,0)^\top$ and $\Sigma_n=\bm \Sigma$, \textcolor{red}{and $\bm \alpha$ represents the proportion of noise}.  
\textcolor{red}{As in the study in \cite{L2023}, we set the number of training samples to 100 (\( n=100 \)), the number of testing samples to 500 (\( m=500 \)), and the data dimensions (\( p \)) to be 2000, 5000, and 10000.}
We record the averaged results of 100 runnings in Table \ref{Tab3}. 

\begin{table}[h]\small
  \caption{\scriptsize{Experimental results of SFL-PSVM, SFL-SVM and DrSVM on the synthetic datasets. Here, $n$ is the size of training dataset; $m$ is the size of test dataset; $p$ is the dimension of the input vector; $\alpha$ is the level of the noise; CAR is the classification accuracy rate; NI is the number of iteration; CT is the computation time(s); NTSF is the number of true selected features.}}
  \centering
  \renewcommand{\arraystretch}{1.3}
  \begin{tabular}{cccccccccccccc}
  \Xhline{0.8pt}
     \multirow{2}*{Date Size}  &  \multirow{2}*{$\alpha$} &  \multicolumn{4}{c}{SFL-PSVM}   &  \multicolumn{4}{c}{SFL-SVM}   &  \multicolumn{4}{c}{DrSVM} \\
  \cmidrule(lr){3-6}\cmidrule(lr){7-10}\cmidrule(lr){11-14}
                               &                       &  CAR  &   CT  & NI&  NTSF  &  CAR  &  CT  &  NI  &  NTSF  &  CAR  &  CT  &  NI  &  NTSF   \\
  \hline
  \multirowcell{3}{$p=2000$} &  \textcolor{red}{5\%} & 0.978 & 1.184 & 7 & 10 & 0.972 & 5.513  & 82 & 9.9 & 0.969 & 2.201 & 24 & 10  \\
                                               &  \textcolor{red}{10\%}  & 0.977 & 1.088 & 7 & 10 & 0.970 & 5.678  & 85 & 10  & 0.965 & 2.257 & 26 & 10 \\
                                               & \textcolor{red}{20\%}  & 0.972 & 1.115 & 7 & 10 & 0.965 & 5.721  & 87 & 9.8 & 0.961 & 2.343 & 29 & 9.9\\
  \multirowcell{3}{$p=5000$} &  \textcolor{red}{5\%} & 0.973 & 4.273 & 6 & 10 & 0.967 & 15.527 & 86 & 9.8 & 0.961 & 6.342 & 28 & 9.9\\
                                               &  \textcolor{red}{10\%}  & 0.972 & 4.196 & 6 & 10 & 0.962 & 15.718 & 89 & 9.7 & 0.957 & 6.417 & 32 & 10\\
                                               &  \textcolor{red}{20\%}  & 0.966 & 4.402 & 6 & 10 & 0.957 & 16.396 & 95 & 9.9 & 0.952 & 6.486 & 34 & 10\\
  \multirowcell{3}{$p=10000$} &   \textcolor{red}{5\%} & 0.966 & 10.38 & 6 & 10 & 0.959 & 37.315 & 90 & 10  & 0.956 & 15.19 & 36 & 9.9\\
                                               & \textcolor{red}{10\%}  & 0.962 & 11.02 & 6 & 10 & 0.952 & 46.379 & 96 & 9.9 & 0.950 & 15.22 & 39 & 9.8\\
                                               &  \textcolor{red}{20\%}  & 0.951 & 10.65 & 7 & 10 & 0.947 & 48.284 & 99 & 9.7 & 0.944 & 15.14 & 41 & 10\\
  \Xhline{0.8pt}
  \end{tabular}
\label{Tab3}
\end{table}

\textcolor{red}{Table \ref{Tab3} indicates that SFL-PSVM (the proposed MLADMM) exhibits certain  advantages in terms of computation time and accuracy when compared with SFL-SVM (ADMM in \cite{YX}) and DrSVM (LADMM in \cite{L2023})}.
The main reason for the less computation time of MLADMM is that it has fewer iterations. In addition, the number of true selected features (NTSF) selected by MLADMM is always 10, which is better than  that of   ADMM and LADMM.  \textcolor{red}{For additional experiments on other noises such as $t$, Cauchy, mixed normal, lognormal, and more estimation metrics, we have included them in Table \ref{tab8} in Supplementary Material \ref{C}.}

\textcolor{red}{During the review process, the reviewer raised two questions: one is the reason why ADMM \cite{YX} requires a long computation time, and the other is that LADMM does not need the time shown in Table \ref{Tab3} to calculate DrSVM in  \cite{L2023}. For the first question, there are mainly two reasons. The first one is that ADMM requires at least $p$ more subproblems to iterate than LADMM and MLADMM, which requires more iterative steps to converge the algorithm when $p$ is large. The second one is that ADMM needs to solve the inverse of a $p$-dimensional matrix. Although \cite{YX}  provided the conjugate gradient method for inversion, it still requires a significant computational burden. For the second question, please note that in \cite{L2023}, the simulated production classification data does not have noise at the classification  boundary, which can cause slow convergence of the LADMM algorithm. We can see that the LADMM algorithm here requires more than 20 iteration steps, while in Table 2 of \cite{L2023}, the number of iterations is 7 or 8. Therefore, the data presented in Table \ref{Tab3} is several times larger than Table 2 in  \cite{L2023}. This phenomenon once again reflects that SFL-PSVM is insensitive to noise, so the algorithm we solved, MLADMM, can converge so quickly.}

\subsubsection{Regression}
In the synthetic data, the model for the simulated data is generated from $\boldsymbol{y}_{i}=\boldsymbol{x_i}^\top \boldsymbol{\beta}+{\boldsymbol{\varepsilon} _i}, i=1,2,\dots,n$. We consider two simulations. The first one applies SFL-QRE ($\tau = 0.5$) and SFL-LS to regression where the design matrix is a general numerical matrix, and the other one applies quantile fused Lasso signal approximator (QFLSA) and FLSA to signal approximation (pulse detection) where the design matrix is the identity matrix. We will utilize the proposed MLADMM algorithms and MADMM  \cite{W2023} to solve SFL-QRE, as well as  ADMM  \cite{YX} and  LADMM \cite{LMY} to solve SFL-LS regression. Furthermore, we will apply the MLADMM  to implement QFLSA, 
 and \textbf{flsa} package \cite{H2010} and ADMM  \cite{YX}    to implement  FLSA. Each experiment may include various error terms, such as normal, mixed normal, $t$, and $\text{Cauchy}$ distributions, to evaluate the performance of the methods under different conditions.
\\
\\
$\bullet$  \textbf{Example 1}:  Each row of the design matrix $\boldsymbol{X}$ is generated by $N(\boldsymbol{0},\boldsymbol{\Omega})$ distribution with Toeplitz correlation matrix ${\boldsymbol{\Omega} _{ij}} = {0.5^{\left| {i - j} \right|}}$ and normalized such that each column has $\ell _2$ norm $\sqrt{n}$. To produce sparse and blocky coefficients, following \cite{LMY}, we divide $p$ into 80 groups in order, and randomly selecte 10 groups
denoted as a sample set $\mathcal{A}$ whose cardinality is $s$. $\mathcal{A}^c$ is the complement of set $\mathcal{A}$. Also we set $(n,p,s)=(720,2560,320)$ and the true coefficient vector $\boldsymbol{\beta}^*$ is generated by
\begin{equation}\label{sbeta}
\boldsymbol{\beta}^* = \left\{ \begin{array}{l}
\boldsymbol{\beta} _{\mathcal{A}_i}^*=\text{U}\left[ { - 3,3} \right] \times \boldsymbol{1}_{\mathcal{A}_i},\ \text{if} \ i=1,2,\dots,10,\\
\textbf{0},\ \ \ \ \ \ \ \ \ \ \ \ \ \ \ \ \ \ \ \ \ \ \ \ \ \ \ \text{otherwise}.
\end{array} \right.
\end{equation}
Here, $\text{U} \left[ -3,3 \right]$ denotes  the uniform distribution on the interval $\left[ { - 3,3} \right]$.  In order to evaluate the performance of SFL-QRE and SFL-LS, we define some specific measurements  as  \cite{LMY}:
\begin{enumerate}
\item $\#\{|\beta_i-\beta_i^*|<0.1, i \in \mathcal{A} \}$ represents the number of  $\beta_i$ satisfies $|\beta_i-\beta_i^*|<0.1$ for $i \in \mathcal{A}$, where $\#$ defines the number of the elements in the set;

\item  $\max_{i\in \mathcal{A}}|\beta_i-\beta_i^*|$ represents maximum difference between $\beta_i$ and $\beta_i^*$ for  $i \in \mathcal{A}$;

\item  $\#\{|\beta_i|<0.1, i \in \mathcal{A}^c \}$  represents  the number of $\beta_i$ satisfies $|\beta_i|<0.1$ for $i \in \mathcal{A}^c$;

\item  $\max_{i\in \mathcal{A}^c}|\beta_i|$ represents maximum value of  $\beta_i$ for $i \in \mathcal{A}^c$.

\end{enumerate}
In addition, we also recorded the  number of iterations (NI) and calculation time (CT).
We report the averaged selected results of 100 runnings in Table \ref{Tab4}.
\begin{table}[h]\small
  \caption{\small{Selected results for the SFL-QRE and SFL-LS regression under different noise, and the best result is highlighted in bold}.}
  \centering
  \renewcommand{\arraystretch}{1.3}
  \resizebox{\textwidth}{!}{
  \begin{tabular}{ccccccccc}
  \Xhline{0.8pt}
     Error  &  Methods  & NI & CT & $\#\{|\beta_i-\beta_i^*|<0.1, i \in \mathcal{A} \}$  &  $\max_{i\in \mathcal{A}}|\beta_i-\beta_i^*|$  &   
$\#\{|\beta_i|<0.1, i \in \mathcal{A}^c \}$  &  $\max_{i\in \mathcal{A}^c}|\beta_i|$  \\
  \hline
  \multirowcell{5}{$N(0,1)$}  & MLADMM & \bf 72 & \bf 8.6  &\bf320 & 0.0801 & \bf 2240 & 0.0486  \\
                              & MADMM & 102 &12.4 &\bf 320 & 0.0826 & \bf 2240 & 0.0452  \\
                              & LADMM & 103 &12.1 & \bf 320 & 0.0973 & \bf 2240 & 0.3587  \\
                              & ADMM  & 91 &14.8  &\bf 320 &\bf 0.0742 & \bf 2240 & \bf 0.2614  \\
                              &          &    &      &        &           &          &  \\
  \multirowcell{5}{$t(2)$}    & MLADMM & \bf 68 &\bf 8.1  & \bf 316 & \bf 0.1392 & \bf 2239 & \bf 0.1429  \\
                              & MADMM & 117  & 13.8& 314 & 0.1405 & 2237 & 0.1501  \\
                              & LADMM & 153   & 19.9 & 309 & 0.3762 & 2216 & 0.7134  \\
                              & ADMM  & 139 &17.4 & 306 & 0.2910 & 2223 & 0.6580  \\
                              &          &    &      &        &           &          &  \\
  \multirowcell{5}{Cauchy}    & MLADMM & \bf 84 &\bf 9.1  & \bf 319 & \bf 0.1802 & \bf 2234 & \bf 0.2103  \\
                              & MADMM & 114 & 13.0 &313 & 0.1889 & 2231 & 0.2248  \\
                              & LADMM & 500+ &50.6  & 182 & 1.9300 & 1640 & 1.6920  \\
                              & ADMM & 500+& 53.7 & 177 & 1.8210 & 1710 & 1.7432  \\
                              &          &    &      &        &           &          &  \\
  \multirowcell{4}{$0.9N(0,1)+0.1N(0,25)$}  & MLADMM & \bf 75 &\bf 8.8 & \bf 318 & \bf 0.1082 & \bf 2239 & \bf 0.1097  \\
                                            & MADMM & 97 &12.1 & 317 & 0.1115 & \bf 2239 & 0.1016  \\ 
                              & LADMM & 111& 14.2 & 312 & 0.2418 & 2231 & 0.4691  \\
                              & ADMM & 99& 12.5 & 316 & 0.2720 & 2229 & 0.4708  \\
  \Xhline{0.8pt}
  \end{tabular}}
\label{Tab4}
\end{table}
\begin{figure}[H]
  \centering
  \includegraphics[width=8cm,height=5cm]{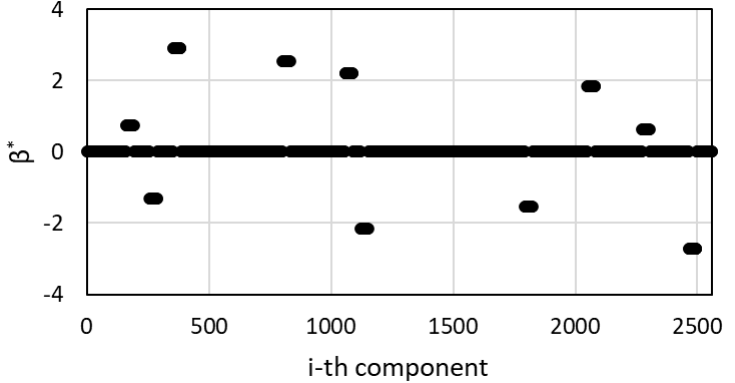}
  \caption{\small{True regression coefficients.}}
  \label{Fig0}
\end{figure}

Table \ref{Tab4} shows that  SFL-LS regression is irreplaceable in the normal distribution. But in the heavy-tailed distribution, its performance is worse than that of SFL-QRE, especially in the Cauchy distribution. We place the true coefficients of the randomly generated regression in Figure \ref{Fig0} and the visualization results of coefficient estimation under various error distributions in Figure \ref{Fig1}. Figure \ref{Fig1} shows that even under the extreme thick tailed distribution of Cauchy, our SFL-QRE recovery signal can still recover well. The recovery results of these coefficients indicate that MLADMM  can effectively solve  SFL-QRE.

\begin{figure}[H]
  \centering
  \includegraphics[width=15cm,height=10cm]{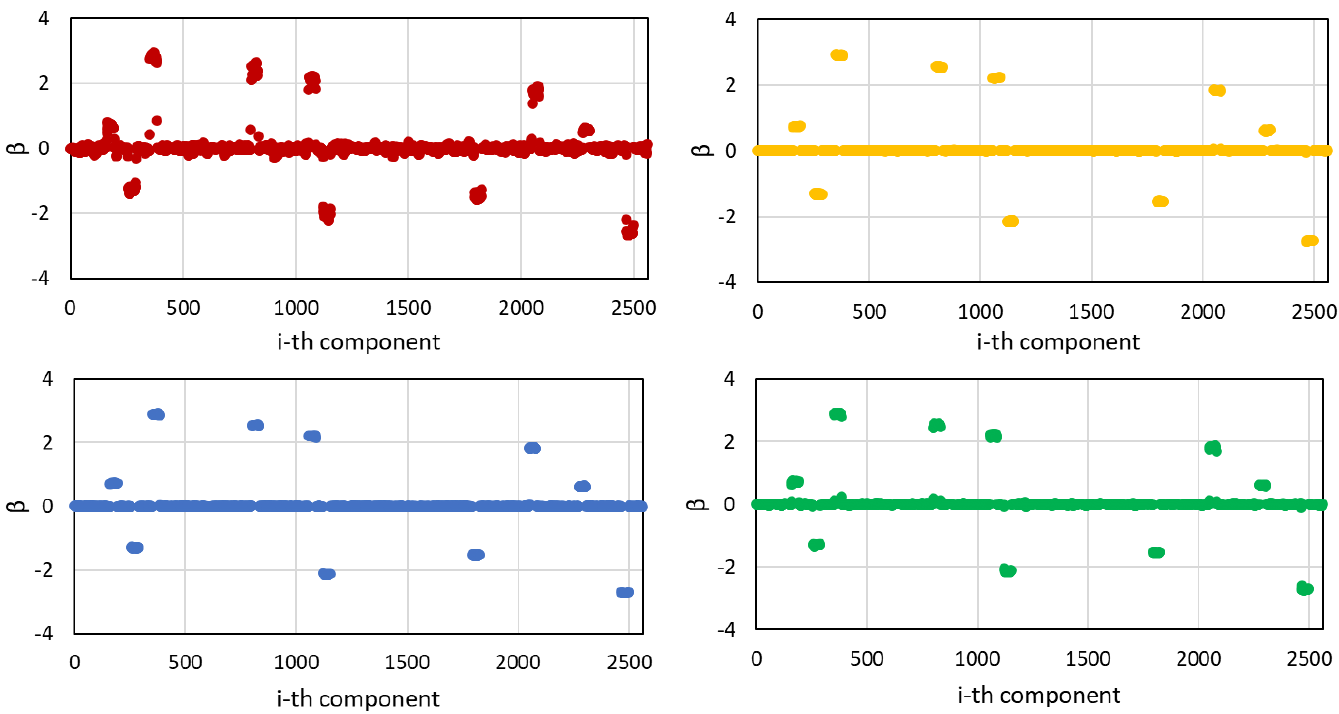}
  \caption{\small{Recovery results for SFL-QRE under various error distributions. Top left: Cauchy distribution. Top right: mixed normal distribution. Bottom left: normal distribution.   Bottom right: $t(2)$ distribution. }}
  \label{Fig1}
\end{figure}

$\bullet$  \textbf{Example 2}:  We consider the problem of estimating pulses of varying width in the presence of different distribution noises.  As in \cite{PS2016}, generate pulses by $\boldsymbol{y}_{i}=\boldsymbol{\beta}_i+ 0.2 {\boldsymbol{\varepsilon} _i}, i=1,2,\dots,n$ and model the pulse signal as sparse piecewise constant. We apply the FLSA and QFLSA to estimate the individual pulses. In order to save space, we only present the results of heavy-tailed error distribution,  and numerical experiments related to normal errors can be found in \cite{PS2016}. 

Figure \ref{Fig2} shows synthetic pulse signal and noisy pulse signal. \textcolor{red}{The first picture in Figure \ref{Fig2}   is the original pulse signal, which is a typical sparse and blocky signal. The second  picture in Figure \ref{Fig2} shows the error of the original pulse signal plus the $t(2)$ distribution. It can be seen that there is a significant difference from the original signal, and the sparse block signal has been contaminated and not obvious. The last  picture in Figure \ref{Fig2} shows the error of the original signal plus the Cauchy distribution. Due to the characteristics of the Cauchy distribution, the absolute value of some noise values may be relatively large, resulting in some large absolute values in the final signal, making certain parts of the signal steep (pay attention to the numerical values of the vertical axis).}

\textcolor{red}{Figure \ref{Fig3} shows  the recovery of pulse signals contaminated by $t$ distribution and Cauchy distribution errors by FLSA and QFLSA. We use two method, namely flsa package  and ADMM, to implement FLSA, and MLADMM proposed in this paper to implement QFLAS. The results in the first column of Figure \ref{Fig3} represent signal recovery contaminated by $t$ distribution and Cauchy distribution using the flsa package. The second column depicts signal recovery using the ADMM algorithm under $t$ distribution and Cauchy distribution contamination, while the third column illustrates signal recovery using the MLADMM algorithm under the same contaminations.}

\textcolor{red}{To compare the performance of these three algorithms in recovering pulse  signals, please first pay attention to the values of the vertical coordinates of the six pictures in Figure 5. Our proposed algorithm MLADMM effectively recovers the numerical range of [-3,4] in the original pulse signal, while FLSA and ADMM algorithms perform similarly, with signals that cannot be recovered falling within this range, especially in the Cauchy distribution.}

To sum up, Figure \ref{Fig3}  show that QFLSA is more effective than FLSA in signal recovery of heavy-tailed noise. In fact, FLSA  has lost the accuracy of signal recovery from $t$ and Cauchy noise   pulse signals.  This indicates that QFLSA has great potential advantages for complex pulse signal recovery.

\begin{figure}[h]
  \centering
  \includegraphics[width=15cm,height=5cm]{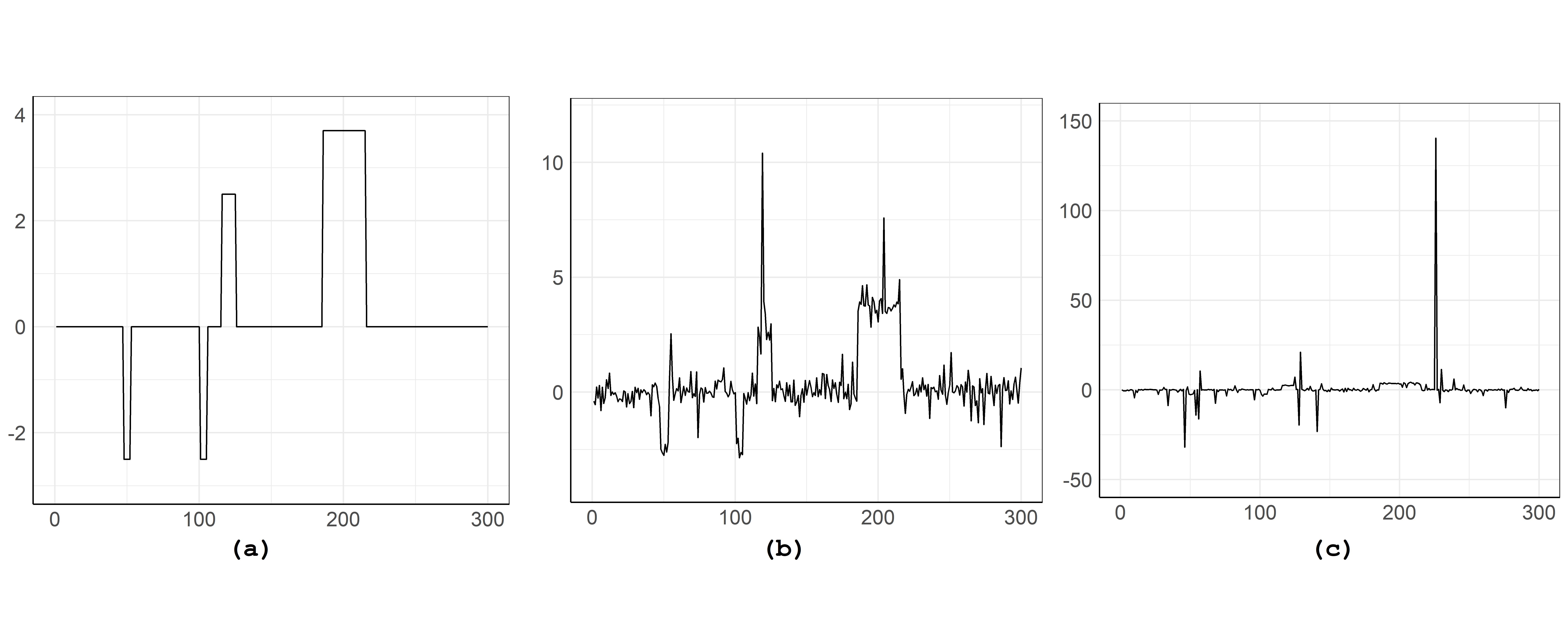}
  \caption{\small{Synthetic pulse signal and noisy pulse signal. (a) synthetic pulse signal, (b) $t$ distribution noise pulse signal, (c) Cauchy distribution noise pulse signal.}}
  \label{Fig2}
\end{figure}

\begin{figure}[H]
  \centering
  \includegraphics[width=15cm,height=9cm]{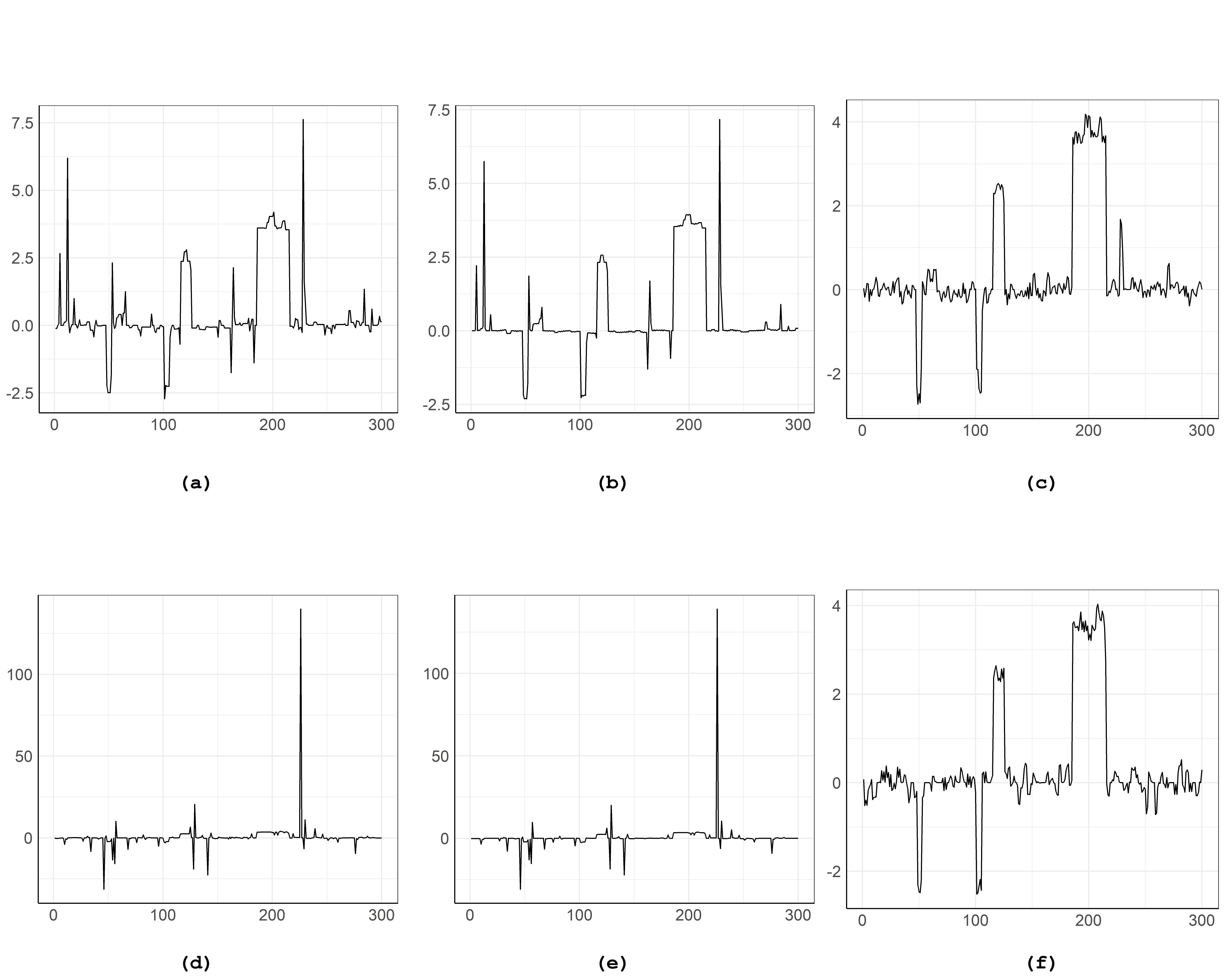}
  \caption{\small{\textcolor{red}{The recovery pulse signals contaminated by $t$ distribution and Cauchy distribution.} The above three figures are estimates of $t$-noise pulse signals, and the following three figures are estimates of Cauchy-noise pulse signals. (a) flsa package for $t$ noise, (b) ADMM for $t$ noise, (c) MLADMM for $t$ noise, (d) flsa package for Cauchy noise, (e)  ADMM for Cauchy noise, (f) MLADMM for Cauchy noise. }}
  \label{Fig3}
\end{figure}

\subsection{Real Datasets}
In the following  real datasets, the first three are microarray gene expression datasets for classification, the last dataset is brain tumor dataset for signal recovery. 
\subsubsection{Microarray gene expression dataset}
A microarray gene expression dataset includes thousands of genes (features) and dozens of observations, that is, genes selection in this section is a high dimensional problem with  $p>>n$. We implement  the proposed SFL-PSVM,  SFL-SVM in \cite{YX}  and DrSVM  in \cite{L2023}  on three microarray gene expression datasets, including \textit{colon}, \textit{duke breast} and \textit{leukemia}.  The datasets are available at \url{https://www.csie.ntu.edu.tw/~cjlin/libsvmtools/datasets/}. The concise description of three datasets is presented in Table \ref{Tab5}, with each dataset randomly partitioned into two distinct subsets, serving as training and test datasets. Table \ref{Tab6} reports the experimental results of SFL-PSVM, SFL-SVM and DrSVM  on the three datasets.

\begin{table}[H]\small
  \caption{\scriptsize{Summary of the three microarray gene expression datasets: class is the number of classes; $n$ is the size of the
training dataset; $m$ is the size of the test data; and $p$ is the dimension of the input vector.}}
  \centering
  \renewcommand{\arraystretch}{1.3}
  \setlength{\tabcolsep}{10mm}{
  \begin{tabular}{cccccccccccccc}
  \Xhline{0.8pt}
     Data set & class & $n$ & $m$ & $p$ \\
  \hline
     Colon & 2 & 31 & 31 & 2000 \\
     Duke breast & 2 & 22 & 22 & 7129 \\
     Leukemia & 2 & 38 & 34 & 7129 \\
  \Xhline{0.8pt}
  \end{tabular}}
\label{Tab5}
\end{table}

\begin{table}[h]\small
  \caption{\scriptsize{Experimental results of SFL-PSVM, SFL-SVM and DrSVM  on the real datasets. CAR, CT and NI have the same meaning as in Table \ref{Tab3};  NNC is the number of nonzero coefficients.}}
  \centering
  \renewcommand{\arraystretch}{1.3}
  \begin{tabular}{ccccccccccccc}
  \Xhline{0.8pt}
  \multirow{2}*{Dateset} & \multicolumn{4}{c}{QSF-SVM} & \multicolumn{4}{c}{SFL-SVM} & \multicolumn{4}{c}{DrSVM}\\
  \cmidrule(lr){2-5}\cmidrule(lr){6-9}\cmidrule(lr){10-13}
  &  CAR  &   CT  &NI &NNC&  CAR  &  CT   &  NI & NNC &  CAR  &  CT   &  NI &  NNC   \\ \hline
  Colon    & 0.896 & 0.629 & 5 & 21 & 0.818 & 24.33 & 312 & 182 & 0.903 & 0.549 & 7 & 498 \\ 
  Duke breast & 0.851 & 6.919 & 5 & 10 & 0.806 & 49.60 & 152 & 93 & 0.977 & 3.041 & 5 & 94 \\ 
  Leukemia   & 0.882 & 7.010 & 6 & 5  & 0.826 & 53.27 & 149 & 102 & 0.912 & 2.964 & 5 & 154 \\
  \Xhline{0.8pt}
  \end{tabular}
\label{Tab6}
\end{table}

According to Table \ref{Tab6}, our algorithm  for SFL-PSVM (MLADMM) is more effective than that of SFL-SVM  (ADMM \cite{YX}).  This indicates that pinball (quantile) loss is more applicable in microarray gene expression datasets compared to hinge loss.
 DrSVM (LADMM \cite{L2023})  has the best performance in CAR (classification accuracy rate) and CT (computation time). 
\textcolor{red}{The main reason why SFL-PSVM does not perform as well as DrSVM in terms of CAR performance may be that these classification coefficients do not have a block like structure, but rather a structure of correlation between variables exists. For  CT performance, in fact, LADMM requires fewer iterations of variables than MLADMM, and even if LADMM has more iterations, it requires less iteration time. In addition, SFL-PSVM and SFL-SVM have better sparsity compared to DrSVM, mainly due to the presence of total variation, which may compress the differences between some adjacent coefficients, which may result in coefficients with smaller absolute values estimated around 0 being compressed to 0.}

\subsubsection{Brain tumor dataset}
Our second application of QFLSA is to analyze a version of the comparative genomic hybridization (CGH) data from \cite{BBJHVRS}, which was further studied by  \cite{TP,WLY}. This version of the dataset can be found in the \textbf{cghFLasso} package in \textbf{R}. It is obvious that CGH data is a group of noisy signals with dimensions of $990 \times 1$. The dataset contains CGH measurements from 2 glioblastoma multiforme (GBM) brain tumors. CGH array experiments are usually applied to estimate each gene's DNA copy number by obtaining the $\log_2$ ratio of the number of DNA copies of the gene in the tumor cells relative to the number of DNA copies in the reference cells. Mutations to cancerous cells lead to amplification or deletions of a gene from the chromosome, so the goal of the analysis is to identify these gains or losses in the DNA copies of that gene. This dataset is visualized in Figure \ref{Fig4}. It is clear that it is not symmetrical data about the origin. For more accurate information on this dataset, one can refer to \cite{BBJHVRS}.

\begin{figure}[h]
  \centering
  \includegraphics[width=12cm,height=4cm]{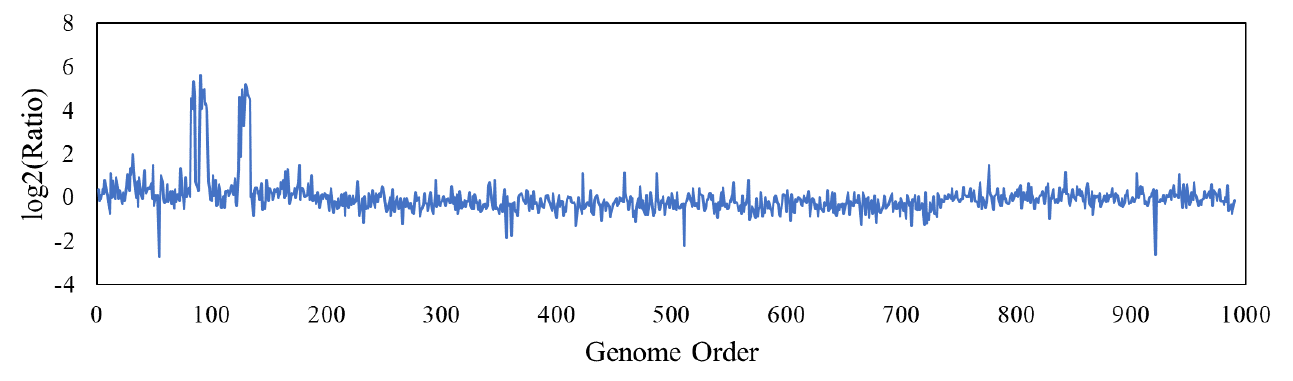}
  \caption{Original CGH data. \label{Fig4}}
  \label{Fig5}
\end{figure}
We use  FLSA  and  QFLSA ($\tau$) to fit CGH data.
\textcolor{red}{The $\tau$ parameter in QFLSA offers flexibility in handling data asymmetry, unlike FLSA, which is tailored for symmetric data due to its least squares loss function. We varied $\tau$ from 0.1 to 0.9 and used fitting  errors ($\text{FE} = \| \bm y - \hat{\bm y} \|_1 / n$) as measurement indicators, where $\hat{\bm y}$ represents the estimated values for both FLSA and QFLSA. 
\begin{figure}[h]
  \centering
  \includegraphics[width=9cm]{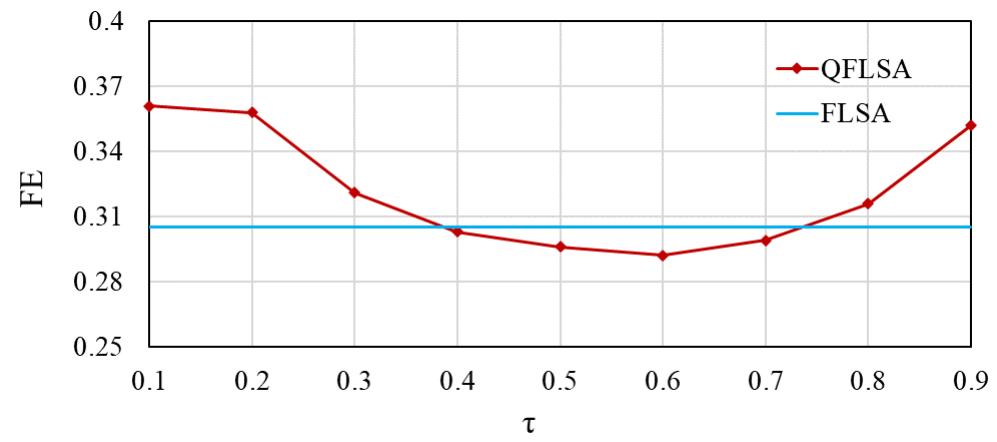}
  \caption{The fitting  errors of CGH data by FLSA and QFLSA.\label{Fig7}}
  \label{Fig6}
\end{figure}
In Figure \ref{Fig7}, we illustrate the prediction errors of FLSA and QFLSA across different $\tau$ levels. The findings show that at $\tau=0.4$, $0.5$, $0.6$, and $0.7$, QFLSA prediction error is consistently smaller than FLSA, indicating QFLSA superior fit for asymmetric data.
Additionally, in Figure \ref{Fig5}, we depict the fitting results of QFLSA on GCH data at $\tau=0.6$, showcasing its sparsity and blocky characteristics.}

\begin{figure}[h]
  \centering
  \includegraphics[width=12cm,height=4cm]{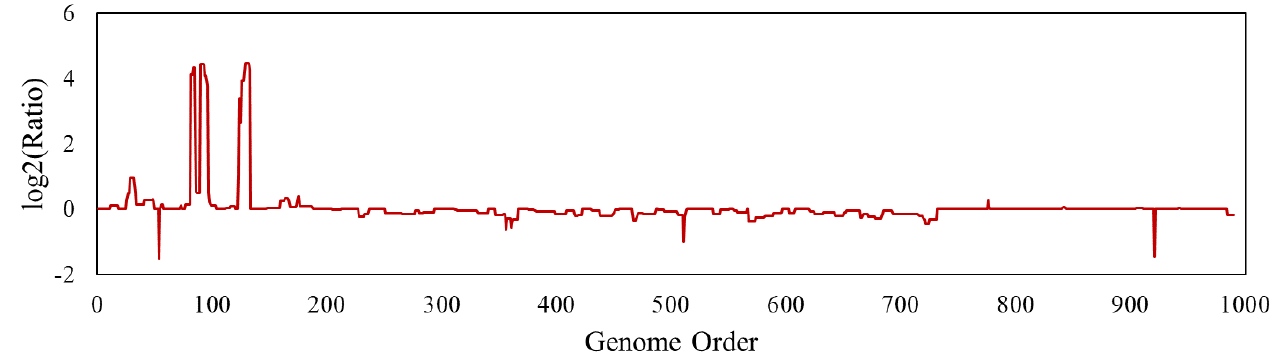}
  \caption{The fitting of CGH data by QFLSA.\label{Fig5}}
  \label{Fig6}
\end{figure}

\section{Conclusion}\label{sec7}
\textcolor{red}{In this paper,  we combine the quantile (pinball) loss and sparse fused Lasso regularization to propose a novel sparse classification model (SFL-PSVM), which is a generalization of the traditional sparse fused Lasso SVM (SFL-SVM). Compared to  the traditional  SFL-SVM classification model, SFL-PSVM exhibits insensitivity to decision boundary noise due to the pinball loss.  In the specific application of high-dimensional data classification, if the coefficients are  sparse and blocky, SFL-PSVM will be a good and robust alternative to SFL-SVM.
In addition, we find that there is a unified optimization form for quantile regression with sparse fused Lasso (SFL-REG) between SFL-PSVM.}

\textcolor{red}{To tackle the unified optimization problem, we introduce the multi-block linearized alternating direction method of multipliers (MLADMM). This algorithm significantly reduces computational complexity compared to similar methods for solving sparse fused Lasso models, especially in high-dimensional scenarios. In numerical experiments, while maintaining estimation accuracy, our computation time is slashed by approximately 40\% compared to the multi-block ADMM calculation of the SFL-REG model. Moreover, compared to the multi-block ADMM algorithm for SFL-SVM computation, our algorithm achieves a time reduction of at least 6-fold. We also demonstrate the convergence of our proposed algorithm and derived its linear convergence rates. Additionally, our algorithm is versatile, capable of extending to solve various existing and future classification and regression problems incorporating sparse fused Lasso, including hinge loss, least squares loss, square root loss, and $\varepsilon$-insensitive loss functions. We anticipate that the flexibility and efficiency of this algorithm will have significant implications in engineering, particularly in domains necessitating sparse fused Lasso regularization, such as signal processing for images and audio.}

\textcolor{red}{Our algorithm also has some limitations, and we will discuss these limitations below and improve them in future work.} The twin support vector machine, positioned as an alternative to traditional SVMs, offers theoretical advantages in terms of both computational speed and accuracy. Our algorithm  requires some modifications to tackle twin support vector machines with the pinball loss function.
Moreover, the difference of convex (DC) method facilitates the transformation of certain nonconvex penalty terms into the difference between two convex functions. Our forthcoming research endeavors aim to extend the DC method within MLADMM algorithms to accommodate nonconvex losses and/or regularizations.
In recent years, distributed and parallel ADMM algorithms have demonstrated efficacy in addressing large-scale regression problems, showcasing promising performance in numerical experiments. Extending the parallel multi-ADMM approach to sparse fused Lasso problems presents an appealing prospect for managing extensive sample data.

\section*{CRediT authorship contribution statement}
Xiaofei Wu: Conceptualization, Methodology, Software, Data curation, Visualization, Writing-Original draft preparation. Rongmei Liang: Conceptualization, Data curation, Visualization, Writing-Reviewing and Editing. Zhimin Zhang: Conceptualization, Writing-Reviewing and Editing. Zhenyu Cui: Conceptualization, Supervision, Writing-Reviewing and Editing, Funding acquisition. All authors have read and agreed to the published version of the manuscript.

\section*{Acknowledgements}
We appreciate the insightful and constructive comments and advice from the associate editor and three anonymous reviewers, which greatly improve the quality of this paper. Specifically, we are deeply grateful to Professor Bingsheng He from Nanjing University for his numerous insightful discussions, which greatly inspired and encouraged us to utilize ADMM algorithms for solving statistical learning problems. The research of Zhimin Zhang was supported by the National Natural Science Foundation of China [Grant Numbers 12271066, 12171405, 11871121], and the research of Xiaofei Wu and Rongmei Liang was supported by the Scientific and Technological Research Program of Chongqing Municipal Education Commission [Grant Numbers KJQN202302003].

\begin{footnotesize}

\end{footnotesize}

\newpage
\section*{Supplementary materials}
\appendix
\section{Supplementary numerical experiments}\label{C}
One reviewer suggested adding noise other than Gaussian noise at the decision boundary of the classification, while another reviewer recommended incorporating some additional estimation  metrics. Next, we supplement with some suggested experiments to enhance the completeness of this paper. For noise, we have added the effects of $t (2)$ distribution, Cauchy distribution, mixed normal, and lognormal on the SVM decision function.  Considering that most of these errors have the property of origin symmetry, the $\tau$ parameter of SFL-PSVM here is selected as 0.5. For estimation metrics, we have added evaluation metrics such as mean square error (MSE), classification accuracy rate (CAR), false negatives (FN, where non-zero coefficients are estimated to be 0), false positives (FP, where zero coefficients are not estimated to be 0), calculation time (CT), and number of iterations (NI). The data in the following two tables are the average values of 100 independent numerical experiments.

In Table \ref{tab7}, we present supplementary numerical experiments on other types of noise for Example 2 in Section \ref{sec41}. Note that $\alpha$ is the proportion of noise to the total sample data. To conserve space, we only include cases with $n = 200$ and the application of $t(2)$ distribution and Cauchy distribution noise in the table. Results for other situations not presented here are similar. For the estimation of $\beta_0$ (true value is 0) and $\bm\beta$ (true value is 2.5), SFL-PSVM  shows smaller bias compared to SFL-SVM estimation. Regarding MSE and CAR, which indicate estimation and prediction effects, SFL-PSVM demonstrates certain advantages over SFL-SVM. Additionally, the CT and NI metrics of SFL-PSVM are significantly better than those of SFL-SVM. These numerical results indicate that SFL-PSVM is more effective and robust than SFL-SVM when dealing with low dimensional data in the presence of noise at the decision boundary.
\begin{table}[!ht]\small
    \caption{\small{Comparison of two SFL-SVMs low dimensional data classification under $t(2)$ and Cauchy noise.}}
   \centering
   \renewcommand{\arraystretch}{1.3}
   \setlength{\tabcolsep}{3mm}{
   \begin{tabular}{cllllllllllll}
    \hline
    t(2) & Method & $\beta_0$ & $ \beta_1$ & MSE & CAR & CT & NI\\ \hline
    \multirowcell{2}{$\alpha = 5\%$} & SFL-SVM & 0.072 & 2.345 & 0.014 & 0.973 & 1.121 & 57   \\ 
        & SFL-PSVM &  0.050 &  2.381 &  0.010 &  0.995 &  0.113 &  12  \\ 
    \multirowcell{2}{$\alpha = 10\%$} & SFL-SVM & -0.093 & 2.194 & 0.052 & 0.963 & 1.139 & 62 \\ 
        & SFL-PSVM & 0.062 & 2.356 & 0.012 & 0.991 & 0.112 & 10 \\ 
    \multirowcell{2}{$\alpha = 20\%$} & SFL-SVM & 0.098 & 2.110 & 0.083 & 0.958 & 1.248 & 69 \\ 
        & SFL-PSVM & 0.059 & 2.312 & 0.016 & 0.987 & 0.114 & 11 \\ \hline
    Cauchy & Method & $\beta_0$ & $\beta_1$ & MSE & CAR & CT & NI \\ \hline
    \multirowcell{2}{$\alpha = 5\%$} & SFL-SVM & 0.077 & 2.339 & 0.016 & 0.970 & 1.132 & 52 \\ 
        & SFL-PSVM & 0.047 & 2.371 & 0.009 & 0.991 & 0.115 & 11  \\
    \multirowcell{2}{$\alpha = 10\%$} & SFL-SVM & 0.079 & 2.165 & 0.061 & 0.965 & 1.247 & 58\\
        & SFL-PSVM & 0.045 & 2.356 & 0.012 & 0.984 & 0.127 & 13 \\
    \multirowcell{2}{$\alpha = 20\%$} & SFL-SVM & 0.085 & 2.158 & 0.091 & 0.960 & 1.309 & 71 \\
        & SFL-PSVM & 0.051 & 2.299 & 0.018 & 0.981 & 0.121 & 12\\ \hline
    \end{tabular}}
    \label{tab7}
\end{table}

In Table \ref{tab8}, we present supplementary numerical experiments for Example 3 in Section \ref{sec41}. To conserve space, we only include cases with $n = 100, m = 500, p=10000$ and the application of $t(2)$ distribution and Cauchy distribution noise in the table. Results for other situations not presented here are similar. In addition to the existing CAR, CT, NI  in Table \ref{Tab3}, We have added FP and FN indicators instead of NTSF to better demonstrate the variable selection ability of the two SFL-SVMs. The results of Table \ref{tab8} indicate that, SFL-PSVM has better variable selection ability than SFL-SVM in ultrahigh dimensional data under various noises. On the test dataset, SFL-PSVM also has better prediction accuracy compared to SFL-SVM. In addition, SFL-PSVM has a significant advantage over SFL-SVM in terms of iteration times, resulting in less computation time. These numerical results indicate that SFL-PSVM is more effective and robust than SFL-SVM when dealing with  ultrahigh dimensional data in the presence of noise at the decision boundary.

\begin{table}[!ht]\small
  \centering
  \caption{\small{Comparison of two SFL-SVMs ultrahigh dimensional data classification under $t$ and Cauchy noise.}}
  \renewcommand{\arraystretch}{1.3}
  \label{tab8}
  \setlength{\tabcolsep}{5mm}{
  \begin{tabular}{cllllll}
    \hline
    t(2) & Method & FN & FP & CAR & CT & NI  \\ \hline
    \multirowcell{2}{$\alpha = 5\%$} & SFL-SVM & 0.13 & 0.05 & 0.951 & 36.10 & 93  \\ 
        & SFL-PSVM & 0.00 & 0.00 & 0.963 & 11.23 & 7  \\ 
    \multirowcell{2}{$\alpha = 10\%$} & SFL-SVM & 0.25 & 0.07 & 0.944 & 40.12 & 101  \\
        & SFL-PSVM & 0.00 & 0.00 & 0.961 & 12.33 & 8  \\ 
    \multirowcell{2}{$\alpha = 20\%$} & SFL-SVM & 0.37 & 0.11 & 0.939 & 46.33 & 109  \\ 
        & SFL-PSVM & 0.00 & 0.00 & 0.960 & 11.19 & 7  \\ \hline
    Cauchy & Method & FN & FP & CAR & CT & NI  \\ \hline
    \multirowcell{2}{$\alpha = 5\%$} & SFL-SVM & 0.18 & 0.10 & 0.930 & 36.55 & 96  \\ 
        & SFL-PSVM & 0.00 & 0.00 & 0.959 & 12.47 & 8  \\ 
    \multirowcell{2}{$\alpha = 10\%$} & SFL-SVM & 0.32 & 0.19 & 0.922 & 41.09 & 103  \\ 
        & SFL-PSVM & 0.00 & 0.00 & 0.951 & 12.33 & 8  \\ 
    \multirowcell{2}{$\alpha = 20\%$} & SFL-SVM & 0.49 & 0.36 & 0.918 & 47.85 & 112  \\ 
        & SFL-PSVM & 0.00 & 0.00 & 0.943 & 13.47 & 9  \\ \hline
    \end{tabular}}
\end{table}

\section{Proofs of Convergence Theorems}\label{A}
Recalling the  constrained optimization (\ref{qflasso_con1}), it is clear that the object function consists of three individual blocks. Take $\theta_1(\tilde{\boldsymbol\beta})=\lambda_1 \|\boldsymbol \beta \|_1$,  $\theta_2(\boldsymbol b)=\lambda_2 \|\boldsymbol b\|_1$ and  $\theta_3(\boldsymbol r)= \frac{1}{n}\sum\limits_{i = 1}^n {\rho_\tau({\boldsymbol{r}_i})}$.  Then, the constrained optimization (\ref{qflasso_con1}) can be rewritten as
\begin{equation}\label{threeblock}
\begin{array}{l}
\min \limits_{\boldsymbol{\beta,b,r}}\left\{ \theta_1(\tilde{\boldsymbol \beta})+\theta_2(\boldsymbol b)+\theta_3(\boldsymbol r)\right\} \\
\text{s.t.}\boldsymbol{A_1}\tilde{\bm \beta}+\boldsymbol{A_2 b}+\boldsymbol{A_3 r}=\boldsymbol c,
\end{array}
\end{equation}
where $\tilde{\boldsymbol \beta}=(\beta_0,\boldsymbol \beta^\top)^\top$ and $${\boldsymbol{A}_1} = \left[\begin{array}{*{20}{c}}
\boldsymbol \gamma & \boldsymbol{X}\\
\boldsymbol{0}_{p-1} & \boldsymbol{F}
\end{array} \right],{\boldsymbol{A}_2} =\left[\begin{array}{*{20}{c}}
\boldsymbol{0}\\
-\boldsymbol{I}_{p-1}
\end{array} \right],{\boldsymbol{A}_3} = \left[\begin{array}{*{20}{c}}
\boldsymbol{I}_n\\
\boldsymbol{0}
\end{array} \right],\boldsymbol{e} = \left[\begin{array}{*{20}{c}}
\bar{\boldsymbol{y}}\\
\boldsymbol{0}_{p-1}
\end{array} \right].$$
It is easy to verify that
\begin{equation}\label{ort-con2}
\boldsymbol{A_2}^\top \boldsymbol{A_3}=\boldsymbol 0. 
\end{equation}
For convenience of description, we take that $\mu=\mu_1=\mu_2$ and $\bm z = (\boldsymbol{u}^\top,\boldsymbol{v}^\top)^\top$.
Take  $\boldsymbol{G}=\eta \boldsymbol{I}_{p+1}-\tilde{{\boldsymbol{X}}}^\top \tilde{\boldsymbol{X}} $, where $\eta>\rho(\tilde{{\boldsymbol{X}}}^\top \tilde{\boldsymbol{X}})$ and $\rho(\tilde{{\boldsymbol{X}}}^\top \tilde{\boldsymbol{X}})$ represents the maximum eigenvalue of $\tilde{{\boldsymbol{X}}}^\top \tilde{\boldsymbol{X}}$. Then, the subproblems of LADMM can be written as
\begin{small}
\begin{equation}\label{3priupdate}
\left\{\begin{array}{l}
\tilde{\boldsymbol{\beta}}^{k+1}=\mathop{\arg \min }\limits_{\tilde {\boldsymbol \beta}} \left\{ \theta_1(\tilde{\boldsymbol{\beta}})- (\boldsymbol{z}^{k})^\top \boldsymbol{A_1}\tilde{\boldsymbol{\beta}}+{\mu \over 2}\|\boldsymbol{A_1}\tilde{\boldsymbol{\beta}}+\boldsymbol{A_2}\boldsymbol{b}^{k}+\boldsymbol{A_3}\boldsymbol{r}^{k}-\boldsymbol c  \|_2^2   \right\}+ {1 \over 2}\| \tilde{\boldsymbol{\beta}}-\tilde{\boldsymbol{\beta}}^k \|_{\boldsymbol{G}}^2, \\
(\boldsymbol{b}^{k+1}, \boldsymbol{r}^{k+1})= \mathop{\arg \min }\limits_{\boldsymbol{b},\boldsymbol{r}} \left\{\theta_2(\boldsymbol b)+\theta_3(\boldsymbol r) - (\boldsymbol{z}^{k})^\top(\boldsymbol{A_2}\boldsymbol{b}+\boldsymbol{A_3}\boldsymbol{r}) +{\mu \over 2}\|\boldsymbol{A_1}\boldsymbol{\beta}^{k+1}+\boldsymbol{A_2}\boldsymbol{b}+\boldsymbol{A_3}\boldsymbol{r} -\boldsymbol c  \|_2^2   \right\}, \\ 
\boldsymbol{z}^{k+1}=\boldsymbol{z}^{k}-\mu(\boldsymbol{A_1}\tilde{\boldsymbol{\beta}}^{k+1}+\boldsymbol{A_2}\boldsymbol{b}^{k+1}+\boldsymbol{A_3}\boldsymbol{r}^{k+1} -\boldsymbol c).
\end{array}\right.
\end{equation}
\end{small}
Clearly,  (\ref{3priupdate}) is a specific application of the original two-block ADMM by regarding $(\boldsymbol{b}^\top,\boldsymbol{r}^\top)^\top$ as one variable, $[ \boldsymbol{A_2},\boldsymbol{A_3}]$ as one matrix and $\theta_2(\boldsymbol{b})+\theta_3(\boldsymbol{r})$ as one function. Readers can refer to (\ref{primalupdate}) for more intuitive understanding. Existing convergence results for the original two-block ADMM such as those in \cite{BY,BY2}, thus hold for the special case (\ref{3priupdate}) with the orthogonality condition. For completeness, we next give a detailed proof. 
Let $\boldsymbol{w}=(\tilde{\boldsymbol{\beta}}^\top,\boldsymbol{b}^\top,\boldsymbol{r}^\top,\boldsymbol{z})^\top$, $\boldsymbol{M}=[\boldsymbol{A_2},\boldsymbol{A_3}]$, and $n_1=n+p-1$, then the convergence of $\boldsymbol{w}$ can  be summarized in the following proposition.
\begin{prop}\label{prop8}
Assume that $\boldsymbol{w}^{k+1}$ is the new iteration point generated by  Algorithm \ref{alg1}, then we have $ (\boldsymbol{w}^{k+1}-\boldsymbol{w}^* )^\top \boldsymbol{H} (\boldsymbol{w}^{k}-\boldsymbol{w}^{k+1} ) \ge 0$. The following inequality holds naturally  
\begin{equation}\label{iee}
\| \boldsymbol{w}^{k+1}-\boldsymbol{w}^* \|_{\boldsymbol{H}}^{2} \le \| \boldsymbol{w}^{k}-\boldsymbol{w}^* \|_{\boldsymbol{H}}^{2}+ \| \boldsymbol{w}^{k}-\boldsymbol{w}^{k+1} \|_{\boldsymbol{H}}^{2},
\end{equation}
where $\boldsymbol{\vartheta}^*$ is an any optimal value of (\ref{3priupdate}) and $\boldsymbol{H} = \left[ {\begin{array}{*{20}{c}}
\boldsymbol{G}&\boldsymbol{0}&\boldsymbol{0}\\
\boldsymbol{0}&{\mu \boldsymbol{M}^\top{\boldsymbol{M}}}&\boldsymbol{0}\\
\boldsymbol{0}&\boldsymbol{0}&{\frac{1}{\mu }{\boldsymbol{I}_{{n_1}}}}
\end{array}} \right].$
\end{prop}
\begin{proof}
For convenience, take $\boldsymbol{\vartheta}=(\boldsymbol{b}^\top,\boldsymbol{r}^\top)^\top$, $\tilde{\theta}_2(\bm \vartheta)= \theta_2(\bm b)+\theta_3(\bm r)$ and $\boldsymbol{\tilde{\vartheta}}=(\tilde{\boldsymbol{\beta}}^\top, \boldsymbol{b}^\top,\boldsymbol{r}^\top)^\top$, $\tilde{\theta}(\tilde {\bm \vartheta})=\theta_1(\tilde{\bm \beta}) + \theta_2(\bm b)+\theta_3(\bm r)$. The subproblems of LADMM in (\ref{3priupdate})  can be reformulated as
\begin{equation}\label{th4op}
\left\{ \begin{array}{l}
\tilde{\boldsymbol{\beta}}^{k+1}= \mathop {\arg \min }\limits_{\tilde{\boldsymbol \beta}}\left\{\theta_1(\tilde{\boldsymbol{\beta}})+{\mu \over 2}\|\boldsymbol{A_1}\tilde{\boldsymbol{\beta}}+\boldsymbol{M}\boldsymbol{\vartheta}^k-\boldsymbol{c}-\boldsymbol{z}^k/\mu  \|_2^2+ {1 \over 2}\|\boldsymbol{\beta}-\boldsymbol{\beta}^k \|_{\boldsymbol{G}}^2  \right\},\\%
\boldsymbol{\vartheta}^{k+1}=\mathop {\arg \min }\limits_{\boldsymbol \vartheta}\left\{ \tilde{\theta}_2(\boldsymbol{\vartheta})+{\mu \over 2}\|\boldsymbol{A_1}\tilde{\boldsymbol{\beta}}^{k+1}+\boldsymbol{M}\boldsymbol{\vartheta}-\boldsymbol{c}-\boldsymbol{z}^k/\mu \|_2^2 \right\},\\%
\boldsymbol{z}^{k+1}=\boldsymbol{z}^{k}-\mu(\boldsymbol{A_1}\boldsymbol{\beta}^{k+1}+ \boldsymbol{M}\boldsymbol{\vartheta}^{k+1} -\boldsymbol c).
\end{array}\right.
\end{equation}
For the $\tilde{\boldsymbol \beta}$-subproblem of (\ref{th4op}), it follows from the convexity of $\theta_1(\tilde{\boldsymbol{\beta}})$ that
\begin{equation}\label{th4beta}
\theta_1(\tilde{\boldsymbol{\beta}})-\theta_1(\tilde{\boldsymbol{\beta}}^{k+1})+(\tilde{\boldsymbol{\beta}}-\tilde{\boldsymbol{\beta}}^{k+1})^\top \left[\mu \boldsymbol{A_1}^{\top}(\boldsymbol{A_1}\tilde{\boldsymbol{\beta}}^{k+1}+\boldsymbol{M}\boldsymbol{\vartheta}^k-\boldsymbol{c}-\boldsymbol{z}^k/\mu )+ \boldsymbol{G}(\tilde{\boldsymbol{\beta}}^{k+1}-\tilde{\boldsymbol{\beta}}^{k}) \right] \ge 0.
\end{equation}
For the $\boldsymbol{\vartheta}$-subproblem of (\ref{th4op}), by the convexity of $\tilde{\theta}_2(\boldsymbol{\vartheta})$, we have
\begin{equation}\label{th4var}
\tilde{\theta}_2(\boldsymbol{\vartheta})-\tilde{\theta}_2(\boldsymbol{\vartheta}^{k+1})+(\boldsymbol{\vartheta}-\boldsymbol{\vartheta}^{k+1})^\top \left[ \mu \boldsymbol{M}^\top (\boldsymbol{A_1}\tilde{\boldsymbol{\beta}}^{k+1}+\boldsymbol{M}\boldsymbol{\vartheta}^{k+1}-\boldsymbol{c}-\boldsymbol{z}^k/\mu)    \right] \ge 0.
\end{equation}
For the $\boldsymbol{z}$-subproblem of (\ref{th4op}), we have
\begin{equation}\label{th4z}
(\boldsymbol{z}-\boldsymbol{z}^{k+1})^\top\left[(\boldsymbol{z}^{k+1}-\boldsymbol{z}^{k})/\mu+ (\boldsymbol{A_1}\tilde{\boldsymbol{\beta}}^{k+1}+ \boldsymbol{M}\boldsymbol{\vartheta}^{k+1} -\boldsymbol c)\right] \ge 0.
\end{equation}
Summing the above three inequalities together, we obtain
\begin{equation}\label{vi}
\begin{split}
\theta(\boldsymbol{\tilde{\vartheta}})&-\theta(\boldsymbol{\tilde{\vartheta}}^{k+1})+(\boldsymbol{w}-\boldsymbol{w}^{k+1})^\top F(\boldsymbol{w}^{k+1})+ (\tilde{\boldsymbol{\beta}}-\tilde{\boldsymbol{\beta}}^{k+1})^\top \mu \boldsymbol{A_1}^\top \boldsymbol{M}(\boldsymbol{\vartheta}^k-\boldsymbol{\vartheta}^{k+1}) \\
&+ (\tilde{\boldsymbol{\beta}}-\tilde{\boldsymbol{\beta}}^{k+1})^\top \boldsymbol G (\tilde{\boldsymbol{\beta}}^{k+1}-\tilde{\boldsymbol{\beta}}^{k})+{1 \over \mu}(\boldsymbol{z}-\boldsymbol{z}^{k+1})^\top(\boldsymbol{z}^{k+1}-\boldsymbol{z}^{k}) \ge 0,
\end{split}
\end{equation}
where $F(\bm w^{k+1})=\left[ {\begin{array}{*{20}{c}}
-\bm A_1^\top \bm z^{k+1}\\
-\bm M^\top \bm z^{k+1}\\
\boldsymbol{A_1}\tilde{\bm{\beta}}^{k+1}+\boldsymbol{M}\boldsymbol{\vartheta}^{k+1}-\boldsymbol{c}
\end{array}} \right]$.

Add $(\boldsymbol{\vartheta}-\boldsymbol{\vartheta}^{k+1})^\top \mu \boldsymbol{M}^\top \boldsymbol{M}(\boldsymbol{\vartheta}^k-\boldsymbol{\vartheta}^{k+1}) $ to both sides of inequality (\ref{vi}), then 
\begin{equation}\label{vi2}
\begin{split}
\theta(\boldsymbol{\tilde{\vartheta}})-\theta(\boldsymbol{\tilde{\vartheta}}^{k+1})+(\boldsymbol{w}-\boldsymbol{w}^{k+1})^\top F(\boldsymbol{w}^{k+1})&+ \mu {\left( \begin{array}{l}
\boldsymbol{\beta}  - {\boldsymbol{\beta} ^{k + 1}}\\
\boldsymbol{\vartheta}  - {\boldsymbol{\vartheta} ^{k + 1}}
\end{array} \right)^\top}\left( \begin{array}{l}
\boldsymbol{A_1}^\top\\
\boldsymbol{M}^\top
\end{array} \right)\boldsymbol{M}(\boldsymbol{\vartheta}^k-\boldsymbol{\vartheta}^{k+1})\\
&\ge (\boldsymbol{w}-\boldsymbol{w}^{k+1})^\top \boldsymbol{H} (\boldsymbol{w}^k-\boldsymbol{w}^{k+1}),
\end{split}
\end{equation}
where $\boldsymbol{H}= \left( {\begin{array}{*{20}{c}}
\boldsymbol{G}&\boldsymbol{0}&\boldsymbol{0}\\
\boldsymbol{0}&{\mu \boldsymbol{M}^\top{\boldsymbol{M}}}&\boldsymbol{0}\\
\boldsymbol{0}&\boldsymbol{0}&{\frac{1}{\mu }{\boldsymbol{I}_{{n_1}}}}
\end{array}} \right).$

It is clear that 
\begin{equation}\label{th4e1}
\theta(\boldsymbol{\tilde{\vartheta}}^*)-\theta(\boldsymbol{\tilde{\vartheta}}^{k+1})+(\boldsymbol{w}^*-\boldsymbol{w}^{k+1})^\top F(\boldsymbol{w}^{k+1})=\theta(\boldsymbol{\tilde{\vartheta}}^*)-\theta(\boldsymbol{\tilde{\vartheta}}^{k+1})+(\boldsymbol{w}^*-\boldsymbol{w}^{k+1})^\top F(\boldsymbol{w}^{*}) \le  0
\end{equation}
and 
\begin{equation}\label{th4e2}
\mu {\left( \begin{array}{l}
\boldsymbol{\beta}^*  - {\boldsymbol{\beta} ^{k + 1}}\\
\boldsymbol{\vartheta}^*  - {\boldsymbol{\vartheta} ^{k + 1}}
\end{array} \right)^\top}\left( \begin{array}{l}
\boldsymbol{A_1}^\top\\
\boldsymbol{M}^\top
\end{array} \right)\boldsymbol{M}(\boldsymbol{\vartheta}^k-\boldsymbol{\vartheta}^{k+1})= (\boldsymbol{z}^{k+1}-\boldsymbol{z}^{k})^\top \boldsymbol{M} (\boldsymbol{\vartheta}^k-\boldsymbol{\vartheta}^{k+1}).
\end{equation}
Let $\boldsymbol{\tilde{\vartheta}}=\boldsymbol{\tilde{\vartheta}}^*, \boldsymbol{w}=\boldsymbol{w}^*, \boldsymbol{\beta}=\boldsymbol{\beta}^*, \boldsymbol{\vartheta}=\boldsymbol{\vartheta}^*$ and $\boldsymbol{z}=\boldsymbol{z}^*$ and bring (\ref{th4e1}) and (\ref{th4e2}) into (\ref{vi2}), then we get
\begin{equation}
(\boldsymbol{w}^{k+1}-\boldsymbol{w}^{*})^\top \boldsymbol{H} (\boldsymbol{w}^k-\boldsymbol{w}^{k+1}) \ge (\boldsymbol{z}^{k}-\boldsymbol{z}^{k+1})^\top \boldsymbol{M} (\boldsymbol{\vartheta}^k-\boldsymbol{\vartheta}^{k+1}).
\end{equation}
From (\ref{th4var}), we obtain the following two inequalities
\begin{align}
\tilde{\theta}_2(\boldsymbol{\vartheta}^k)-\tilde{\theta}_2(\boldsymbol{\vartheta}^{k+1})+(\boldsymbol{\vartheta}^k-\boldsymbol{\vartheta}^{k+1})^\top (-\boldsymbol{M}^\top \boldsymbol{z}^{k+1}) \ge 0,\\
\tilde{\theta}_2(\boldsymbol{\vartheta}^{k+1})-\tilde{\theta}_2(\boldsymbol{\vartheta}^{k})+(\boldsymbol{\vartheta}^{k+1}-\boldsymbol{\vartheta}^{k})^\top (-\boldsymbol{M}^\top \boldsymbol{z}^{k}) \ge 0.
\end{align}
These two inequalities point out that
\begin{equation}
 (\boldsymbol{z}^{k}-\boldsymbol{z}^{k+1})^\top \boldsymbol{M} (\boldsymbol{\vartheta}^k-\boldsymbol{\vartheta}^{k+1}) \ge 0.
\end{equation}
Then, we obtain
\begin{equation}\label{pro6}
(\boldsymbol{w}^{k+1}-\boldsymbol{w}^{*})^\top \boldsymbol{H} (\boldsymbol{w}^k-\boldsymbol{w}^{k+1}) \ge 0.
\end{equation}
Note that for any $\boldsymbol{x}, \boldsymbol {y}$, if $\boldsymbol{y}^\top \boldsymbol{H}(\boldsymbol {x}-\boldsymbol {y}) \ge 0$, then we have 
\begin{equation}\label{key}
\|\boldsymbol{y} \|_{\boldsymbol{H}}^2 \le \|\boldsymbol{x} \|_{\boldsymbol{H}}^2- \|\boldsymbol{x}-\boldsymbol{y} \|_{\boldsymbol{H}}^2.
\end{equation}
Therefore, 
$$\| \boldsymbol{w}^{k+1}-\boldsymbol{w}^* \|_{\boldsymbol{H}}^{2} \le \| \boldsymbol{w}^{k}-\boldsymbol{w}^* \|_{\boldsymbol{H}}^{2}+ \| \boldsymbol{w}^{k}-\boldsymbol{w}^{k+1} \|_{\boldsymbol{H}}^{2}.$$ 
So far, we have completed the proof of Proposition \ref{prop8}. 
\end{proof}

Proposition \ref{prop8} has the same conclusion as Lemma 3 in \cite{LMY}, then the assertions summarized in the following corollary are trivial based on the inequality (\ref{iee}). 
\begin{cor}\label{cor1}
Let ${\bm w^k}$ be the sequence generated by Algorithm \ref{alg1}. Then we have
\vspace{-1em}
\begin{enumerate}
\item $\mathop {\lim }\limits_{k \to \infty}\|\bm w^k-\bm w^{k+1}\|_{\bm H}=0$.
\item The sequence  $\{\bm w^k\}$ is bounded.
\item For any optimal solution $\bm w^*$, the sequence $\{\|\bm w^k-\bm w^* \|_{\bm H} \}$ is monotonically non-increasing.
\end{enumerate}
\end{cor}

Using Corollary \ref{cor1}, like the proof of Theorem 1 in \cite{LMY}, we can prove $\bm w^k $ converges to $\bm w^*$.  The $O(1/k)$ convergence rate in a non-ergodic sense can directly be obtained from the conclusion of \cite{BY2}. Therefore, Theorem \ref{TH1} is proved.

\section{Generalization of Algorithms}\label{B}
Sparse fused Lasso (SFL) has been incorporated into various loss functions for regression tasks, including least squares (LS) \cite{TSRZ}, square root \cite{JLD}, least absolute deviations (LAD) \cite{LTZ},  and quantile loss \cite{W2023}. For classification tasks, SFL has been used with hinge loss \cite{TSRZ}. 
 Recalling the constrained optimizations (\ref{qflasso_con1}), the difference between our proposed model and  existing models is the loss function. Then, we just need to replace quantile loss function with  least squares loss function $\|\boldsymbol r\|_2^2$, hinge loss function $\sum_{i=1}^{n} \max(0, \boldsymbol{r}_i)$, square root loss function $\|\boldsymbol r\|_2$ to achieve this extention. It is clear that we only need to modify the update of $\boldsymbol{r}$-subproblem in MLADMM iteration (\ref{primalupdate}). To implement these diverse regression and classification tasks, we can follow the same steps outlined in the Appendix  C of \cite{W2023}.

Here, we extend above the proposed algorithm to solve SFL-SVR (sparse fused lasso support vector  regression). The loss function of SVR is called the $\varepsilon$-insensitive loss function, defined as follows
\begin{equation}\label{svr}
    L_{\varepsilon}( r_i) = \begin{cases}
        0 & \text{if } |r_i| \leq \varepsilon, \\
        |r_i| - \varepsilon & \text{otherwise}.
    \end{cases}
\end{equation} 
where  $\varepsilon$ is a specified tolerance value.

Before discussing modifications of the  $\boldsymbol{r}$-subproblem in (\ref{svr}), we introduce the proximal operator which help us express its conveniently. 

$\bullet$ The minimization problem 
\begin{align}\label{svr2}
\mathop {\arg \min }\limits_{{r}_i} \frac{1}{n} L_{\varepsilon}( r_i) + \frac{\mu }{2}\left\| {{r}_i -{r}_i^0} \right\|_2^2
\end{align}
with $\mu > 0$ and a given constant ${r}_i^0$, has a closed-form solution defined as
\begin{align}
r_i^*=
\begin{cases}
r_i^0, & \text{if } |r_i^0| < \varepsilon + \frac{1}{n\mu}, \\
\text{sign}({r}_i^0) \varepsilon, & \text{if } |{r}_i^0| \in [\varepsilon, \varepsilon + \frac{1}{n\mu}], \\
\text{sign}(r_i^0)\max\{0, |r_i^0| - \frac{1}{n\mu} \}, & \text{if } |r_i^0| > \varepsilon + \frac{1}{n\mu},
\end{cases}
\end{align}
\begin{proof}
(1) If $|r_i| < \varepsilon$, the objective function in (\ref{svr2}) is
\begin{equation}\label{proof1}
F(r)= \frac{\mu }{2}\left\| {{r}_i -{r}_i^0} \right\|_2^2.
\end{equation}
Then, we have 
\begin{equation}
r_i^*= {r}_i^0. 
\end{equation}
Please note that at this case, it is necessary to have $|{r}_i^0| < \varepsilon$.

(2) If  $|r_i| > \varepsilon$,  the objective function is
\begin{equation}\label{proof2}
F(r)= \frac{|r_i| - \varepsilon}{n} +  \frac{\mu }{2}\left\| {{r}_i -{r}_i^0} \right\|_2^2. 
\end{equation}
According to (\ref{opera1}), we have 
\begin{equation}\label{svr3}
{r}_i^* = \text{sign}({{r}^0_i}) \cdot \text{max}\{ 0,\left| {{{r}^0_i}} \right| - \frac{1}{n\mu} \}.
\end{equation}
In this case, it is necessary to have $|{r}_i^0| > \varepsilon +  \frac{1}{n\mu}$.

Moving forward, our primary focus will be on identifying the optimal solution within the interval $|{r}_i^0| \in [\varepsilon, \varepsilon + \frac{1}{n\mu}]$. It's worth noting that the objective functions in equations (\ref{proof1}) and (\ref{proof2}) share a quadratic form. Consequently, the minimization problem defined in equation (\ref{svr2}) only involves these two objective functions. Therefore, the optimal solution within the interval of interest must be either $ \text{sign}({r}_i^0)\varepsilon$ or $\text{sign}({r}_i^0) (\varepsilon + \frac{1}{n\mu})$. In the following discussion, we will only discuss the case where ${r}_i^0>0$, and the rest of the cases  are similar.

To determine the optimal solution, we need to compare the values of the objective functions at these two points. It is clear that $F(\varepsilon) = \frac{\mu}{2}\left\|\varepsilon - {r}_i^0\right\|_2^2$, and $F(\varepsilon + \frac{1}{n\mu}) = \frac{1}{n^2\mu} + \frac{\mu}{2}\left\|\varepsilon + \frac{1}{n\mu} - {r}_i^0\right\|_2^2$. Since it follows that $ \frac{1}{n\mu} > 0 $ and $\varepsilon - {r}_i^0 \in [ -\frac{1}{n\mu},0]$, we have $F(\varepsilon + \frac{1}{n\mu}) > F(\varepsilon)$.
Therefore, if ${r}_i^0 \in [\varepsilon, \varepsilon + \frac{1}{n\mu}]$, the optimal solution is ${r}_i^* = \varepsilon$. 

Based on the above discussion, we arrive at the following expression for $r_i^*$:
\begin{align}
r_i^*=
\begin{cases}
r_i^0, & \text{if } |r_i^0| < \varepsilon + \frac{1}{n\mu}, \\
\text{sign}({r}_i^0) \varepsilon, & \text{if } |{r}_i^0| \in [\varepsilon, \varepsilon + \frac{1}{n\mu}], \\
\text{sign}(r_i^0)\max\{0, |r_i^0| - \frac{1}{n\mu} \}, & \text{if } |r_i^0| > \varepsilon + \frac{1}{n\mu}.
\end{cases}
\end{align}
\end{proof}

Then, to solve SFL-SVR using the MLADMM algorithm, we only need to modify equation (\ref{admmu_r}) to the following equation 
\begin{align}
r_i^{k+1}=
\begin{cases}
r_0^k, & \text{if } |r_0^k| < \varepsilon + \frac{1}{n\mu}, \\
\text{sign}({r}_0^k) \varepsilon, & \text{if } |r_0^k| \in [\varepsilon, \varepsilon + \frac{1}{n\mu}], \\
\text{sign}(r_0^k)\max\{0, |r_0^k| - \frac{1}{n\mu} \}, & \text{if } |r_0^k| > \varepsilon + \frac{1}{n\mu},
\end{cases}
\end{align}
where $r_0^k = {\bar{\boldsymbol{y}}_i} - (\boldsymbol{X\beta} ^{k + 1})_i -\beta_0^{k+1} \bm \gamma_i + \frac{{\boldsymbol{u}_i^k}}{{{\mu _1}}}  - \frac{\tau }{{n{\mu _1}}} $.

Recently, Ye et al.\cite{YGS} proposed a  generic quadratic nonconvex $\varepsilon$-insensitive loss function for robust support vector regression. The function is defined as follows:
\[
L_{\varepsilon,t}(r_i)=
\begin{cases}
\left({t-\varepsilon} \right)^{2} + s|r_i| - st, & \text{if } |r_i|>t, \\
\left(|r_i|-\varepsilon\right)^2, & \text{if } \varepsilon \leq |r_i| \leq t, \\
0, & \text{if } |r_i|<\varepsilon,
\end{cases}
\]
where $r_i$ represents the training error of the $i$-th training sample, $t$ is an elastic interval parameter, and $s$ is a robustification parameter. Considering that the loss function is quadratic in nature, we can derive the closed-form solution of its proximal operator using the previously mentioned derivation method. However, it is important to note that this function is non-convex and therefore cannot be included in the proof framework of our algorithm's convergence. In our future work, we plan to design a theoretically guaranteed ADMM algorithm specifically tailored for this robust SVR case.

\end{document}